%% file: Parabolic_f.tex
\documentclass[a4paper, 11pt]{amsart}
\usepackage{longtable, stackrel, multicol, fancybox, tcolorbox}
\usepackage{bm,amsfonts,amsmath,amssymb,mathrsfs,bbm,amsthm}
\usepackage{graphicx, color,hyperref,eurosym, eucal,palatino, marginnote}

\definecolor{darkblue}{rgb}{0.2,0.2,0.6}
\definecolor{darkblue2}{rgb}{0.2,0.2,0.9}
\definecolor{superdarkblue}{rgb}{0.2,0.2,0.3}
\hypersetup{colorlinks=true, linkcolor=darkblue,citecolor=darkblue,
urlcolor = superdarkblue}
\definecolor{citegreen}{rgb}{0.2,0.2,0.6}

\usepackage{geometry}

\input commands.tex

\makeatother

\begin{document}

\title[The Dirichlet Laplacian on a generalized parabolic layer]{
\textsc{Spectral asymptotics of the Dirichlet Laplacian on a generalized parabolic layer}}

\author{Pavel Exner}
\address{Department of Theoretical Physics, Nuclear Physics Institute, Czech Academy of Sciences, 25068 \v Re\v z near Prague, Czechia, and Doppler Institute for Mathematical Physics and Applied Mathematics, Czech Technical University, B\v rehov\'a 7, 11519 Prague, Czechia}
\email{exner@ujf.cas.cz}

\author{Vladimir Lotoreichik}
\address{Department of Theoretical Physics, Nuclear Physics Institute, 	Czech Academy of Sciences, 25068 \v Re\v z, Czech Republic}
\email{lotoreichik@ujf.cas.cz}

\begin{abstract}
We perform quantitative spectral analysis of the self-adjoint Dirichlet Laplacian $\Op$ on an unbounded, radially symmetric (generalized) parabolic layer $\cP\subset\dR^3$. It was known before that $\Op$ has an infinite number of eigenvalues below the threshold of its essential spectrum. In the present paper, we find the discrete spectrum asymptotics for $\Op$ by means of a consecutive reduction to the analogous asymptotic problem for an effective one-dimensional Schr\"o\-dinger operator on the half-line with the potential the behaviour of which far away from the origin is determined by the geometry of the layer $\cP$ at infinity.
\end{abstract}

\maketitle

\section{Introduction}

\subsection{State of the art and motivation}
The topic of this paper is the 
spectral asymptotics for the Dirichlet Laplacian in domains of the form of a generalized parabolic layer. By the layer, we understand the Euclidean domain consisting of the points, whose distance from a hypersurface, in general unbounded one, is less than some given constant. The self-adjoint Laplace operator on unbounded layers subject to the Dirichlet boundary condition has been considered in numerous papers, 
\eg~\cite{CEK04, DOR15, DEK01, DLO17,  ET10, KL14, LL07, LR12}, see also~\cite[Chap. 4]{EK} and the references therein.

Unbounded layers belong to the class of \emph{quasi-cylindrical domains} in the classification~\cite[Dfn.~X.6.1\,(ii)]{EE} introduced  by Glazman, because they do not contain arbitrarily large cubes, but they do contain infinitely many disjoint cubes of a fixed size. For this reason, as a rule, the essential spectrum is non-empty and its threshold is strictly positive; \cf~\cite{DEK01, KL14} and~\cite[Chap.~4]{EK}. A challenging question is to prove the existence of bound states below this threshold and to find their properties. Various general results on the bound states can be found in~\cite{CEK04, DEK01, LL07, LR12} and in~\cite[Chap.~4]{EK}. However, there is no result general enough to claim that this problem is completely understood. In this perspective, more precise analysis of various particular cases deserves attention. Examples of layers treated so far include mildly curved layers~\cite{BEGK01, EKr01}, conical layers~\cite{DOR15, ET10, OP17}, and so-called octant (or Fichera) layers~\cite{DLO17}.

In the present paper, we initiate an investigation of \emph{generalised parabolic layers} by performing the analysis of radially symmetric layers of this kind. With the indicated layer geometry in mind, the previously known general results yield an explicit formula for the threshold of the essential spectrum and imply the existence of infinitely many bound states below it. Our aspiration here is to understand better the nature of this discrete spectrum and to obtain its quantitative properties.

\subsection{Geometric setting}\label{ssec:geom}
Here and in what follows, we use the shorthand notations $\dR_+ := (0,\infty)$, $\cdR_+ := [0,\infty)$, and $\dI_a := (-a,a)$ for $a > 0$. We also introduce conventional Cartesian coordinates $\xx = (x_1,x_2,x_3)$ in the Euclidean space $\dR^3$.

Given the parameters $k > 0$, $\aa > 1$, and $R \ge 0$,
we consider the class of real-valued functions $f\in \cC^\infty(\cdR_+)$
such that $f(0) = \df(0) = 0$ and that
$f(x) = kx^\aa$ holds for all $x \ge R$.
The \emph{generalized symmetric paraboloid} determined by such a function $f$ is defined by
\begin{equation}\label{eq:Sigma}
	\Sg = \Sg(f) := \left\{\big( \xx, f(|\xx|) \big)\in\dR^3 \colon \xx\in\dR^2 \right \}.
\end{equation}
The principal curvatures of $\Sg$ will be denoted by $\kp_1,\kp_2$.
Note that the usual paraboloid corresponds to the choice $\aa = 2$ and $R = 0$ in the above definition. 

The configuration space of our problem is a fixed-width
layer built over $\Sg$, in other words, the $a$-tubular neighborhood of $\Sg$ defined
as\footnote{\distance}
\begin{equation}\label{eq:cP}
	\cP = \cP(f, a) := \big\{ \xx\in\dR^3 \colon \dist\big( \xx, \Sg(f)\big) < a \big\}.
\end{equation}
The domain $\cP = \cP(f,a)$ will be called generalized radially symmetric parabolic layer,
for brevity we will speak of a \emph{generalized parabolic layer} as there is no danger of confusion.

In the following, $\nu(\xx)$ stands for the unit
normal vector to the surface $\Sg$ at the point $(\xx, f(|\xx|))$.
The mapping $\dR^2\ni\xx \mapsto \nu(\xx)$ is smooth
and the orientation of $\nu(\xx)$ is fixed so that its projection on the $x_3$-axis is positive
for $|\xx| \ge R$.

In Subsection~\ref{ssec:Sigma}, we show the existence of a constant $a_\star = a_\star(f) > 0$
satisfying $a_\star\|\kp_j\|_\infty < 1$
for $j=1,2$ 
such that the restriction of the map
\begin{equation}\label{eq:cL}
	\cL\colon \dR^2\times \dR\arr \dR^3,
	\qquad
	\cL(\xx,u) = (\xx, f(|\xx|))^\top + \nu(\xx) u,
\end{equation}
onto $\dR^2\times \dI_a$ is injective for all $a < a_\star$. Consequently,
the generalized parabolic layer of a half-width $a < a_\star$ can alternatively be represented as
the image of $\dR^2\times\dI_a$ under the map $\cL$.
Everywhere in this paper we consider only the generalized parabolic layers the half-width of which satisfies
the condition $a < a_\star$.
%
\subsection{Definition of the operator}
We define a non-negative, symmetric quadratic form on the Hilbert space $L^2(\cP)$ by
\begin{equation}\label{eq:Q_omega_0}
	\sQ [ \psi ] := \| \nb\psi \|_{ L^2(\cP;\dC^3) }^2,
	\qquad
	\dom \sQ := H^1_0(\cP).
\end{equation}
According to~\cite[Lem.~6.1.1 and Lem.~6.1.3]{Dav95}, the quadratic form $\sQ$ is closed and densely defined.
This allows us to introduce the main object of this paper, namely, the self-adjoint operator $\Op$ in $L^2(\cP)$ associated with the form $\sQ$ via the first representation
theorem~\cite[Thm.~VI~2.1]{Ka}.
The operator $\Op$ is nothing but the \emph{Dirichlet Laplacian} on the generalized
parabolic layer $\cP$, the name referring to the fact that the operator $\Op$ acts as follows,
\begin{equation}\label{eq:Op}
\begin{split}
	\Op\psi = -\Dl\psi,	
	\qquad
	\dom\Op = \big\{ \psi\in H^1_0(\cP) \colon \Dl\psi \in L^2(\cP) \big\}.
\end{split}
\end{equation}

\subsection{Notation and the main result}
Before stating the main result of this paper,
we introduce a few further notations. For measurable functions
$\sff_1,\sff_2\colon \dR_+\arr \dR_+$
satisfying $\lim_{E\arr 0^+}\sff_j(E)=  +\infty$,
$j=1,2$, and $\lim_{E\arr 0^+} \frac{\sff_1(E)}{\sff_2(E)} = 1$, we agree to write
\[
	\sff_1(E)\underset{E\searrow 0}{\sim} \sff_2(E).
\]
Furthermore, let $\sfT$ be the semi-bounded self-adjoint operator associated with a quadratic form $\sT$.
We denote by $\sess(\sfT)$ and $\sd(\sfT)$ the essential and the discrete spectrum of $\sfT$, respectively;
by $\s(\sfT)$ we denote the (entire) spectrum of $\sfT$, \ie $\s(\sfT) = \sess(\sfT)\cup\sd(\sfT)$.
We set $\lme(\sfT) := \inf\sess(\sfT)$ and, for a given $j\in\dN$, $E_j(\sfT)$
denotes the $j^{\text{th}}$-eigenvalue of $\sfT$ in the interval $(-\infty, \lme(\sfT))$.
These eigenvalues are arranged in the ascending order and repeated according to their multiplicities.
We define the counting function of $\sfT$ as
\[
	\cN_E(\sfT)
	:=
	\#\big\{j\in\dN \colon E_j(\sfT) < E\big\}, \qquad E \le \lme(\sfT).
\]	
When working with the quadratic form $\sT$, we use the notations $\sess(\sT)$, $\sd(\sT)$, $\s(\sT)$,
$\lme(\sT)$, $E_j(\sT)$, and $\cN_E(\sT)$  instead.
If it is clear from the context, we will not indicate the dependence of the eigenvalues
and the threshold of the essential spectrum on the operator, \ie we will write
$E_j$ and $\lme$ instead of $E_j(\sfT)$ and $\lme(\sfT)$, respectively.

First, we recall spectral properties of the Hamiltonian $\Op$,
which follow from the previously known general results after checking the appropriate
geometric assumptions.
\begin{prop}\label{prop:known}
	For any $a \in\big (0, a_\star\big)$, the following claims hold.
	\begin{myenum}
		\item
		$\sess(\Op) = [\lme,\infty)$, where $\lme = \left (\frac{\pi}{2a}\right )^2$.
		\item $\#\sd(\Op) = \infty$.
	\end{myenum}
\end{prop}
To be specific, one has to verify that both the principal curvatures of $\Sg$ are bounded,
and that the mean and the Gauss curvatures of $\Sg$ vanish at infinity, then the item~(i)
of the above proposition follows from~\cite[Prop. 4.2.1]{EK}; see also~\cite{KL14}.
In addition, checking that
the total Gauss curvature of $\Sg$ does not vanish and making use of the rotational symmetry for $\cP$ we
derive item~(ii) of the above proposition from~\cite[Cor. 4.2.2]{EK}.

In this paper we give an alternative proof of the fact that $\#\sd(\Op) = \infty$
and, what is more important, we obtain as the main result the spectral asymptotics of $\Op$.
\begin{thm}\label{thm:main}
	For any $a \in \big(0, a_\star\big)$,
	the counting function of $\Op$ satisfies
	\[
		\cNE(\Op)
		\underset{E\searrow 0}{\sim}
		\frac{1}{2\pi}
		\frac{\aa k}{2^\aa}
		\frac{\sfB\left (\frac32,\frac{\aa}{2}- \frac12  \right ) }{E^{\frac{\aa}{2} -\frac12}},
	\]	
	where $\sfB(\cdot,\cdot)$ is the Euler beta function.	
\end{thm}
\begin{remark}
In the particular case when the layer is `genuinely' parabolic, $\alpha=2$, possibly outside a compact, the asymptotic expression on the right-hand side simplifies to $\frac{k}{8\sqrt{E}}$. We note that for larger values of $\alpha$ the accumulation is slower. Moreover, using Stirling formula \cite[6.1.37]{Abramowitz-Stegun} we find that
\[
	\lim_{\aa\arr\infty}
	\frac{1}{2\pi}
	\frac{\aa k}{2^\aa}
	\frac{\sfB\left (\frac32,\frac{\aa}{2}- \frac12  \right ) }{E^{\frac{\aa}{2} -\frac12}}
	=
	\lim_{\aa\arr\infty}
	\frac{k}{\sqrt{2\pi\alpha}2^\aa E^{\frac{\aa}{2} -\frac12}} = \infty
\]	
holds for any fixed $E < 1/4$. This might
formally correspond to the fact that as $\alpha$ grows, the layer is becoming closer to the one with annular cylindrical end for which the essential spectrum extends below $\lme$; \cf~\cite[Prob.~4.10]{EK}. On the other hand, for the (possibly locally deformed) conical layer, $\aa=1$, where the discrete spectrum is also infinite~\cite{ET10} one might ask whether the accumulation rate is obtained, upon passing to the limit $\aa\to 1^+$ in the above asymptotic expression, but this not the case as it is obvious from the fact that the limit $\frac{k}{4\pi}$ is independent of $E$.
The correct spectral asymptotics for the conical layer~\cite[Thm. 1.4]{DOR15} is given by $\cNE(\Op)\sim\frac{k}{4\pi}|\ln E|$ as $E\arr 0^+$. Thus, an abrupt transition from the power-type accumulation
rate of eigenvalues to an exponential one occurs upon switching from $\aa > 1$ to $\aa = 1$.
\end{remark}
In the proof of Theorem~\ref{thm:main} we first proceed as in~\cite[Thm. 1.4]{DOR15}. Namely, we decompose the Hamiltonian $\Op$ into an infinite orthogonal sum of the fiber operators, acting on the two-dimensional
meridian domain each. In our geometric setting, the meridian domain is the generalized parabolic half-strip.
Such a reduction is possible in view of the rotational invariance of the domain $\cP$ with respect to the $x_3$-axis.
The `s-wave' fiber corresponding to radially symmetric functions turns out to produce an infinite discrete spectrum below the threshold $\lme$, while the remaining fibers in the decomposition of $\Op$ altogether produce at most a finite number of eigenvalues below $\lme$, and thus they do not influence the accumulation rate in the vicinity of the threshold. Recall in this respect the conical layer case, where one is able to make a stronger claim, namely that \emph{none} of the fibers except the `s-wave' one contributes to the discrete spectrum.

The spectral analysis of the `s-wave' fiber is
reduced to the well-known properties of the one-dimensional Schr\"odinger operators
corresponding to the differential expression
$-\frac{\dd^2}{\dd s^2} + q$
on the positive semi-axis,
where the bounded
potential $q\colon\dR_+\arr\dR$ asymptotically behaves as
\[
	q(s) \sim -\frac14\pot\quad\;\text{as}\quad s \arr +\infty.
\]
Note that for the conventional paraboloid ($\aa = 2$), the potential $q$ behaves at infinity as the Coulomb potential.
This reduction is the core of the paper. As an intermediate
step, it contains straightening of the underlying generalized parabolic half-strip
in suitably chosen curvilinear coordinates, in the spirit of~\cite{ES89, DE95}; see also~\cite[Chap. 1]{EK}. This step
is  new compared to the conical layer, where the meridian domain is straight right
from the beginning.

\section{Preliminaries}
In this preparatory section we collect the auxiliary material, which is needed in the proof of the main result.
First, in Subsection~\ref{ssec:fibers}, we provide
the standard decomposition of the self-adjoint operator $\Op$ into fibers.
Furthermore, in Subsection~\ref{ssec:param} we derive the arc-length
parametrization for the graph of $f(\cdot)$,
define its signed curvature
and introduce curvilinear coordinates in the
associated generalized parabolic strip.
Then, in Subsection~\ref{ssec:Sigma},
we check that the generalized paraboloid $\Sg$ satisfies some of the
assumptions in~\cite[Chap. 4]{EK}
obtaining in this way a proof of Proposition~\ref{prop:known}.
Finally, we recall in Subsection~\ref{ssec:1D}
several standard spectral results on a class of one-dimensional Schr\"{o}dinger operators we will need in the following.

\subsection{The fiber decomposition of $\Op$}\label{ssec:fibers}
%
In view of the rotational symmetry, it is convenient to express
the Hamiltonian $\Op$ in the cylindrical coordinates.
Let $\dR^2_+$ be the positive half-plane $\dR_+\tm\dR$.
We consider the cylindrical coordinates $(r, z, \tt) \in \dR_+^2 \tm \dS^1$
linked with the Cartesian coordinates $(x_1, x_2, x_3)$ via
the following conventional relations
\begin{equation}\label{eqn:co_cyl}
	x_1 = r\cos\tt,\qquad x_2 = r\sin\tt,\qquad x_3 = z.
\end{equation}
We denote the graph of the function $x\mapsto f(x)$ by
\begin{equation}\label{eq:Gamma}
	\G = \G(f) := \big\{ ( x, f(x) ) \colon x > 0 \big\} \subset \dR^2_+.
\end{equation}
In order to describe in cylindrical coordinates the generalized parabolic layer $\cP = \cP(f, a)$
of the half-width $a \in (0, a_\star)$, we first define its meridian domain
$\cG = \cG(f,a)\subset\dR^2_+$ by
%
\begin{equation*}
	\cG = \cG(f,a) := \big\{ \xx \in \dR_+^2 \colon \dist\big(\xx,\G(f)\big) < a \big\}.
\end{equation*}
Note that the domain $\cG$ can be viewed as a generalized parabolic half-strip.
The generalized parabolic layer $\cP$
in~\eqref{eq:cP} is defined in
the cylindrical coordinates~\eqref{eqn:co_cyl} by
\begin{equation}
	\cP = \cP(f, a) := \cG(f, a)\tm \dS^1.
\end{equation}
The layer $\cP$ can be seen as a sub-domain of $\dR^3$ constructed via rotation
of the meridian domain $\cG$ around the $x_3$-axis.
For later purposes, we split the boundary $\p\cG$ of $\cG$ into two disjoint parts,
\[
	\p_0 \cG := \{ 0 \} \tm \dI_a
	\and
	\p_1 \cG := \p \cG \sm \ov{ \p_0 \cG}.
\]
Next, we introduce the usual cylindrical $L^2$-space and the first order
cylindrical Sobolev space on $\cP$ as
\[
	L_\cyl^2(\cP) := L^2(\cG\tm\dS^1; r\dd r \dd z \dd \tt),\qquad
	H^1_\cyl(\cP)
	:=
	\big\{ \psi \colon \psi, \p_r\psi, \p_z\psi, r^{-1}\p_\tt \psi \in L_\cyl^2(\cP)\big\}.
\]
The space $H^1_\cyl(\cP)$ is endowed with the norm $\|\cdot\|_{H^1_\cyl(\cP)}$,
defined for all $\psi \in H^1_\cyl(\cP)$ by
\[
	\|\psi\|_{H^1_\cyl(\cP)}^2
	:=
	\|\psi\|_{L^2_\cyl(\cP)}^2	
	+
	\int_\cP\bigg(|\p_r \psi|^2 + |\p_z \psi|^2 + \frac{|\p_\tt \psi|^2}{r^2}\bigg)
	r\dd r \dd z \dd\tt.
\]
The expression for the quadratic form $\sQ$ associated with $\Op$ can be rewritten
in the cylindrical coordinates as follows
\[
	\sQ_\cyl[\psi]
	=
	\int_{\dS^1} \int_\cG
	\bigg(|\p_r\psi|^2 + |\p_z\psi|^2 + \frac{|\p_\tt\psi|^2}{r^2} \bigg)r \dd r\dd z\dd \tt.
\]
Following the strategy of~\cite{DOR15, KLO17, ET10},
we consider an orthonormal basis of the Hilbert space $L^2(\dS^1)$ given by
\begin{equation}\label{eq:vm}
	v_m(\tt) = (2\pi)^{-1/2} e^{\ii m \tt},\qquad m\in\dZ.
\end{equation}
For any $m\in\dZ$, we define the mapping
\begin{equation}\label{eq:projector}
	\Pi_m\colon L^2_\cyl(\cP)\arr L^2(\cG; r \dd r \dd z),
	\qquad
	(\Pi_m\psi)(r,z,\tt) = \big( \psi(r,z,\cdot), v_m \big)_{L^2(\dS^1)}.
\end{equation}
Performing a `partial wave' decomposition~\cite[App. to X.1]{RS2}, see also~\cite{DOR15, KLO17, LO16}, with respect to the basis in~\eqref{eq:vm}, we obtain
\begin{equation}\label{eq:orthogonal_decomposition}
	\Op \cong \bigoplus_{m\in\dZ} \sfF_m,
\end{equation}
where the symbol $\cong$ stands for the unitary equivalence relation
and, for all $m\in\dZ$, the operators $\sfF_m$
acting on $L^2(\cG; r\dd r \dd z)$ are called the \emph{fibers} of~$\Op$.
They are associated through the first
representation theorem with closed, densely defined, symmetric,
and non-negative quadratic forms in
the Hilbert space $L^2(\cG;r\dd r\dd z)$
\begin{equation}\label{eq:forms}
\begin{split}
	\sF_m[\psi] & := \sQ_{\rm cyl}[\psi \otimes v_m] =
	\int_\cG \bigg( |\p_r\psi|^2 + |\p_z\psi|^2 + \frac{m^2}{r^2}|\psi|^2 \bigg)r\dd r\dd z,\\
	\dom \sF_m &  := \Pi_m \big( H^1_0(\cP) \big).
\end{split}
\end{equation}
The domain of the operator $\sfF_m$ can
be deduced from the quadratic form $\sF_m$
in the standard way via the first representation theorem.

Next, we introduce the unitary operator
\[
	\sfU \colon  L^2(\cG;r\dd r\dd z) \arr L^2(\cG),
	\qquad
	\sfU \psi := \sqrt{r} \psi.
\]	
This unitary operator allows us to
transform the quadratic forms $\sF_m$
into equivalent ones expressed in a flat metric. Indeed, the quadratic form
$\sF_m$ is unitarily equivalent via $\sfU$ to the form in the Hilbert space
$L^2(\cG)$ defined as
\begin{equation}\label{eqn:fq_flat}
\begin{split}
	\frf_m [ \psi ]
	&:= \sF_m[\sfU^{-1}\psi] =
	\int_\cG
	\bigg(|\p_r\psi|^2 + |\p_z\psi|^2 + \frac{m^2-\frac14}{r^2}|\psi|^2\bigg) \dd r\dd z,\\
	\dom\frf_m & := \sfU(\dom \sF_m).
\end{split}\end{equation}
In fact, it is not difficult to check that $\cC_0^\infty(\cG)$ is a core for the form $\frf_m$ and that
for all $m\ne 0$ its form domain satisfies
\begin{equation}\label{eq:domain_incl}
	\dom\frf_m = \sfU(\dom \sF_m) = H^1_0(\cG).
\end{equation}
Finally, for the sake of convenience, we introduce for $m\in\dZ$
the following shorthand notation for frequently used integrands:
\begin{equation}\label{eq:EF}
	\cE_m[\psi] := |\p_r\psi|^2 + |\p_z\psi|^2 + \frac{m^2}{r^2}|\psi|^2
	\and
	\cF_m[\psi] := |\p_r\psi|^2 + |\p_z\psi|^2 + \frac{m^2 -\frac14}{r^2}|\psi|^2.
\end{equation}

\subsection{The arc-length parametrization of $\G$ and associated coordinates on $\cG$}
\label{ssec:param}
The arc-length parametrization for the curve
$\G = \G(f)$ in~\eqref{eq:Gamma} is given by
\[
	\cdR_+\ni s \mapsto \big( \phi(s),f(\phi(s))\big )\in\dR^2,
\]
where $\phi \colon \cdR_+\arr \cdR_+$ is a
monotonously increasing function such that $\phi(0) = 0$
and that $\lims\phi(s) = \infty$,
and which satisfies the ordinary differential equation
\begin{equation}\label{eq:ODE_phi}
	\dph^2 \left (1+ \df^2(\phi)\right ) = 1.
\end{equation}
On the interval $[R,\infty)$, the above equation
reduces to $\dph^2 \left (1+ \aa^2k^2\phi^{2\aa-2}\right ) = 1$.
The equation~\eqref{eq:ODE_phi} combined with
$\phi(0) = 0$ and $\df(0) = 0$ immediately implies
\begin{equation}\label{eq:phi_easybnds}
	\dph(0) = 1 \and 0 < \dph \le 1.
\end{equation}
The signed curvature of $\G$ is given by
\begin{equation}\label{eq:signedcurv}
	\gg 	
	=
	\big(\df(\phi)\ddph + \ddf(\phi)\dph^2\big)\dph
	-
	\df(\phi)\dph\ddph 
	=
	\ddf(\phi)\dph^3.
\end{equation}
On the interval $[R,\infty)$, the above expression simplifies to
$\gg = \aa(\aa-1) k\dph^3\phi^{\aa-2}$.
Asymptotic properties of $\phi$ and $\gg$ that
will be needed in the following are obtained in Appendix~\ref{app:curvature}.

The unit tangential vector to $\G$ at the point $(\phi(s),f(\phi(s)))$ is given by
$(\dph(s),\,\df(\phi(s))\dph(s))^\top$.
Consequently, the unit normal vector to $\G$ at the same point
can be expressed as
\begin{equation}\label{eq:nu_G}
	\nu_\G(s) = \big(-\df(\phi(s))\dph(s),\,\dph(s)\big)^\top.
\end{equation}
Next, on the half-strip $\Omg :=  \dR_+ \tm \dI_a$,
we introduce the map $\tau\colon \Omg\arr\cG$ by
\begin{equation}\label{eq:tau_k}
	\tau(s,u) =
		\begin{pmatrix}
		\tau_1(s,u)	\\
		\tau_2(s,u)	
		\end{pmatrix}
	:=
	\begin{pmatrix}	
		\phi(s)\\ f(\phi(s))
	\end{pmatrix} +  u\nu_\G(s)
	=
	\begin{pmatrix}
		\phi(s) - u\df(\phi(s))\dph(s)\\	
		f(\phi(s)) + u\dph(s)
	\end{pmatrix}.
\end{equation}
This map defines convenient curvilinear coordinates $(s,u)$
on the generalized parabolic strip $\cG$.
Note also that the Jacobian $\cJ = \cJ(s,u)$ of the map $\tau$ is given by
\begin{equation}\label{eq:cJ}
\begin{split}
	\cJ& =
	\begin{vmatrix}
	\p_s \tau_1 & \p_u \tau_1\\
	\p_s \tau_2 & \p_u \tau_2
	\end{vmatrix}
	=
	\begin{vmatrix}
	\dph - u\ddf(\phi)\dph^2
	- u\df(\phi)\ddph
	  & -\df(\phi)\dph\\
	\df(\phi)\dph +u\ddph & \dph
	\end{vmatrix}\\[0.4ex]
	& =
	\dph^2 - u\ddf(\phi)\dph^3
		+ \df^2(\phi)\dph^2
	=
	1 - u \gg,
\end{split}	
\end{equation}	
where~\eqref{eq:ODE_phi} was used in the last step.
We also set
\begin{equation}\label{eq:gk}
	g(s,u) := (\cJ(s,u))^2 = (1-u\gg(s))^2.
\end{equation}
%

\subsection{Parametrization and curvatures of $\Sg$}\label{ssec:Sigma}
In this subsection we provide the natural parametrization of $\Sg$
as a surface of revolution in the sense~\cite[\S 3.3, Example 4]{dC}. Using this
parametrization we obtain formul{\ae} for the curvatures of $\Sg$ and a parametrization for the generalized
parabolic layer $\cP$. Finally, we check the assumptions on layers~\cite[Chap. 4]{EK} and
eventually prove Proposition~\ref{prop:known}.

The generalized paraboloid $\Sg$ can be alternatively parametrized as
\begin{equation}\label{eq:s_map}
	\dR_+ \tm \dS^1\ni (s,\tt) \mapsto \s(s,\tt)
	:=
	\big(\phi(s)\cos\tt, \phi(s)\sin\tt, f(\phi(s))\big)^\top,
\end{equation}
where the functions $f$ and $\phi$ are as in Subsection~\ref{ssec:geom}
and Subsection~\ref{ssec:param}, respectively.
According to~\cite[\S 3.3, Example 4]{dC}, see also~\cite[\S 4.2.2]{EK},
the principal curvatures of $\Sg$ at the point $\xx = \s(s,\tt)$ are explicitly given by
\begin{equation}\label{eq:curvatures}
\begin{split}
	\kp_1(s,\tt) &
		= -\frac{ \df(\phi(s))\dph(s) }{ \phi(s) }
		= -\frac{ \df(\phi(s)) }{ \phi(s)(1+\df^2(\phi(s)))^{1/2} },\\
	\kp_2(s,\tt) &
		= -\gg(s)
		= -\ddf(\phi(s))\dph^3(s)
		= -\frac{ \ddf(\phi(s)) }{ \big(1+\df^2(\phi(s)\big)^{3/2} }.
\end{split}	
\end{equation}
Hence, the mean $M = \frac12(\kp_1 + \kp_2)$ and the Gauss
$K  =	\kp_1\kp_2$ curvatures of $\Sg$ can be evaluated as
\begin{equation}\label{key}
\begin{split}
	M(s,\tt) &
	= -\frac{ \df(\phi(s))+\df^3(\phi(s))+ \ddf(\phi(s))\phi(s) }
					{ 2\phi(s)(1+\df^2(\phi(s)))^{3/2} },\\
	K(s,\tt) & = \frac{ \df(\phi(s)) \ddf(\phi(s)) }{ \phi(s) (1+\df^2(\phi(s)))^2 }.
\end{split}
\end{equation}
It follows from~\cite[Sec. 4.2.2]{EK},
using $\lims\dph(s) = 0$ (\cf Proposition~\ref{prop:phi_bnd}\,(ii))
that the total Gauss curvature of $\Sg$ is given by
$\cK  := \int_\Sg K  = 2\pi$. Furthermore, the normal vector $\nu_\Sg(s,\tt)$ to $\Sg$ at $\xx = \s(s,\tt)$
with $(s,\tt) \in \dR_+\tm\dS^1$ can be expressed as
\begin{equation}\label{eq:normal}
	\nu_\Sg(s,\tt) = \left(-\df(\phi(s))\dph(s)\cos\tt, -\df(\phi(s))\dph(s)\sin\tt, \dph(s)\right)^\top,
\end{equation}
thus the layer $\cP$ 
is the image of 
$\dR_+ \tm \dS^1 \tm \dI_a$
under the map
\begin{equation}\label{eq:pi_map}
	\dR_+ \tm \dS^1 \tm \dR\ni (s,\tt,u)\mapsto \pi(s,\tt,u)
		:= \big(	\phi(s)\cos\tt, \phi(s) \sin\tt, f(\phi(s))\big)^\top + u\nu_\Sg(s,\tt).
\end{equation}
Now, we have all the tools to prove Proposition~\ref{prop:known}.
\begin{proof}[Proof of Proposition~\ref{prop:known}]
	First, recall that $\Sg$ is a surface of revolution
	with a non-zero total Gauss curvature $\cK = 2\pi$.
	Furthermore, we immediately notice that both the principal curvatures $\kp_j(s,\tt)$, $j=1,2$,
	of $\Sg$ computed in~\eqref{eq:curvatures} are $\cC^\infty$-smooth on $\cdR_+\times\dS^1$,
	pointwise positive, and vanish in the limit $s\arr\infty$; 
	\cf Lemma~\ref{lem:phi_bnd}
	and Proposition~\ref{prop:gamma}\,(i).
	Consequently, the minimal normal curvature radius of $\Sg$
	defined in~\cite[Chap. 4, Assumption (ii)]{EK} and given by
	\[
		\rho_{\rm m} := \big( \max\{\|\kp_1\|_\infty,\|\kp_2\|_\infty\} \big)^{-1}
	\]
	is positive and finite.
	In addition, the layer $\cP$ is asymptotically planar in the sense of~\cite[\S 4.1.2, Assumption (iii)]{EK},
	because both $M(s,\tt)$ and $K(s,\tt)$ vanish in the limit $s\arr \infty$.

	The remaining step is more technical: we have to check that there exists
	$a_\star \in (0,\rho_{\rm m})$ such that the generalized
	parabolic layer $\cP = \cP(f,a)$ in~\eqref{eq:cP}
	is not self-intersecting for all $a \in (0, a_\star)$ in the sense of~\cite[Chap. 4, Assumption~{\rm (i)}]{EK}.
	The latter is equivalent to showing that
	the restriction of the map $\pi(\cdot)$
	in~\eqref{eq:pi_map} onto $\dR_+ \tm \dS^1 \tm \dI_a$
	is injective for all $a\in (0,a_\star)$.
	Upon checking this last assumption,~\cite[Thm. 4.4]{EK}
	yields $\sess(\Op) = [\lme,\infty)$ and $\#\sd(\Op) = \infty$.
		
    We will proceed by \emph{reductio ad absurdum}. If such an $a_\star > 0$ did  not exist,
    then it would be possible to find $(s_j, \tt_j, u_j)\in\dR_+\tm\dS^1\tm\dI_{\rho_m}$, $j=1,2$,
	such that $(s_1,\tt_1,u_1) \ne (s_2,\tt_2,u_2)$, $s_1\le s_2$, and at the same time
	$\pi(s_1,\tt_1,u_1) = \pi(s_2,\tt_2,u_2)$ with $u_1,u_2 \in (-\rho_{\rm m},\rho_{\rm m})$ having arbitrarily small absolute values.
	Without loss of generality we can assume
	that both $u_1$ and $u_2$ are positive,
	because the surface $\Sg$ splits $\dR^3$
	into its hypograph and epigraph, respectively.
	Using the explicit expression of the normal vector in~\eqref{eq:normal} together with the fact
	that $\pi(s_1,\tt_1,u_1) - \pi(s_2,\tt_2,u_2)$ has trivial projections onto all the axes, we get
	\begin{equation}\label{eq:system}
	\begin{cases}
		\df(\phi(s_1))\cos\tt_1 c_1	- \df(\phi(s_2))\cos\tt_2 c_2 = \phi(s_1)\cos\tt_1 - \phi(s_2)\cos\tt_2,\\
		\df(\phi(s_1))\sin\tt_1 c_1	- \df(\phi(s_2))\sin\tt_2 c_2	= \phi(s_1)\sin\tt_1 - \phi(s_2)\sin\tt_2,\\
		c_1 - c_2 =  f(\phi(s_2)) -	f(\phi(s_1)),
	\end{cases}	
	\end{equation}
	where $c_j := u_j\dot\phi(s_j)$ for $j=1,2$.
	The above system of linear equations with respect to $c_1$ and $c_2$
	has a non-trivial solution if, and only if
	\[
		\Dl :=
		\begin{vmatrix}
		1         				   &  -1                       & f(\phi(s_2)) -	f(\phi(s_1)) \\
		\df(\phi(s_1))\cos\tt_1 & -\df(\phi(s_2))\cos\tt_2  &	\phi(s_1)\cos\tt_1 - \phi(s_2)\cos\tt_2\\
		\df(\phi(s_1))\sin\tt_1 & -\df(\phi(s_2))\sin\tt_2  & \phi(s_1)\sin\tt_1 - \phi(s_2)\sin\tt_2\\
		\end{vmatrix}
		 =	0.
	\]
	The determinant $\Dl$ can be explicitly computed, giving
	\[
	\begin{split}
		\Dl
		& =
		\begin{vmatrix}
		-\df(\phi(s_2))\cos\tt_2  &
		\phi(s_1)\cos\tt_1 -\phi(s_2)\cos\tt_2 \\
		-\df(\phi(s_2))\sin\tt_2  &
		\phi(s_1)\sin\tt_1 - \phi(s_2)\sin\tt_2
		\end{vmatrix}\\
		& \qquad\qquad +
		\begin{vmatrix}
		\df(\phi(s_1))\cos\tt_1  &
		\phi(s_1)\cos\tt_1 -\phi(s_2)\cos\tt_2 \\
		\df(\phi(s_1))\sin\tt_1  &
		\phi(s_1)\sin\tt_1 - \phi(s_2)\sin\tt_2
		\end{vmatrix}\\
		& \qquad\qquad +		
		\big(f(\phi(s_2)) -	f(\phi(s_1))\big)
		\begin{vmatrix}
		\df(\phi(s_1))\cos\tt_1  &
		-\df(\phi(s_2))\cos\tt_2 \\
		\df(\phi(s_1))\sin\tt_1  &
		-\df(\phi(s_2))\sin\tt_2
		\end{vmatrix}
		\\
		& = \sin(\tt_2-\tt_1)
		\Big(
			\df(\phi(s_2))\phi(s_1)	-\df(\phi(s_1))\phi(s_2)
			+
			\df(\phi(s_1))\df(\phi(s_2))
			\big[f(\phi(s_1)) - f(\phi(s_2)\big]
		\Big).
	\end{split}
	\]
	Hence, the requirement $\Dl = 0$ can only be satisfied
	in two disjoint cases:
	\begin{myenum}
		\item $\tt := \tt_1 = \tt_2$ and $s_1 < s_2$;
		\item $\tt_1 \ne \tt_2$
		and $\df(\phi(s_2))\phi(s_1)	-\df(\phi(s_1))\phi(s_2)
			+
			\df(\phi(s_1))\df(\phi(s_2))
			\big[f(\phi(s_1)) - f(\phi(s_2)\big] =0$.
	\end{myenum}
	These cases require separate analysis.
	\smallskip
		
	\noindent\emph{Case (i)}.
	Pick $R_1 > R$ such that
	$\aa k\phi(R_1)^{\aa-1}  > 1$.
	There are two sub-cases to distinguish.
	If $s_1 < R_1$, then for simple geometric reasons there is	a constant $R_2 = R_2(f,R_1) > R_1$
	such that necessarily $s_2 < R_2$ holds.
	Hence,~\cite[Prop.~B.2]{BEHL16},
	applied for the image of $(0,R_2)\tm\dS^1$
	under the mapping in~\eqref{eq:s_map},	
	implies that there exists
	$a_1 \in (0,\rho_{\rm m})$ such that the condition  $u_1,u_2 \in (0,a_1)$ is incompatible
	with $s_1 < s_2$.
	
	In the second sub-case ($s_1 \ge R_1$),
	we reduce the linear system~\eqref{eq:system} to
	\[
	\begin{cases}
		\df(\phi(s_1)) c_1 - \df(\phi(s_2)) c_2 =  \phi(s_1)  -\phi(s_2),\\
		c_1 - c_2 = f(\phi(s_2)) - f(\phi(s_1)).
	\end{cases}
	\]
	Solving it and returning to the initial variables $u_j$, $j=1,2$, we find
	\[
	\begin{split}
		u_1
		& =
		\frac{ \phi(s_1) - \phi(s_2) + \df(\phi(s_2))[f(\phi(s_2)
			-f(\phi(s_1)] }
		{ \dph(s_1)[\df(\phi(s_2)
			-\df(\phi(s_1)]},\\
		u_2 &=
		\frac{ \phi(s_1) - \phi(s_2) + \df(\phi(s_1))[f(\phi(s_2)
			-f(\phi(s_1)] }
		{ \dph(s_2)[\df(\phi(s_2)
			-		\df(\phi(s_1)]}.
	\end{split}	
	\]
	For the sake of brevity, set $x_1 := \phi(R_1)$.
	We estimate $u_1$ using
	$\dph \le 1$ and applying Cauchy's mean-value theorem
	\begin{equation*}\label{eq:u_bound2}
	\begin{split}
		u_1
		& \ge
		\inf_{x\in [x_1 ,\infty)} \frac{\df(x_1)\df(x)-1}{\ddf(x)}
		= \inf_{ x\in [x_1 ,\infty) }
		\frac{\aa^2k^2(x x_1)^{\aa-1}-1}{\aa(\aa-1)kx^{\aa-2} }\\
		&\ge
		\inf_{ x\in [x_1 ,\infty) } 		
		\frac{x(\aa^2k^2 x_1^{\aa-1} - x^{1-\aa})}{\aa(\aa-1)k}
		=
		\frac{x_1(\aa^2k^2 x_1^{\aa-1} - x_1^{1-\aa})}{\aa(\aa-1)k} > 0.
	\end{split}
	\end{equation*}
	Hence, there exists $a_2 \in (0,\rho_{\rm m})$ such that $u_1,u_2\in (0,a_2)$
	is incompatible with $s_1 < s_2$.
	\smallskip
	
	\noindent\emph{Case (ii)}. The linear system~\eqref{eq:system}
	reduces to
	\begin{equation}\label{eq:system_red}
	\begin{cases}
	\df(\phi(s_1))\cos\tt_1 c_1	- \df(\phi(s_2))\cos\tt_2 c_2 = \phi(s_1)\cos\tt_1 - \phi(s_2)\cos\tt_2,\\
	\df(\phi(s_1))\sin\tt_1 c_1	- \df(\phi(s_2))\sin\tt_2 c_2	= \phi(s_1)\sin\tt_1 - \phi(s_2)\sin\tt_2.
	\end{cases}	
	\end{equation}
		Solving it and returning to the initial variables $u_j$, $j=1,2$, we find
	\[
		u_j
		=
		\frac{\phi(s_j)}{\dph(s_j)\df(\phi(s_j))
		},
		\qquad j=1,2.
	\]
	In view of Lemma~\ref{lem:phi_bnd}, we conclude that there exists $a_3\in (0,\rho_{\rm m})$ such that the condition
	$u_1,u_2\in(0,a_3)$ can not be satisfied together with $\tt_1\ne \tt_2$.
	 	
	Merging the outcome of the analysis for the Cases (i) and (ii), we get the claim for $a_\star
		= \min \{a_1,a_2,a_3\}$.
\end{proof}
Note that the principal curvatures $\kp_1,\kp_2$ of $\Sg$ computed in~\eqref{eq:curvatures} do not depend on $\tt$. Using this observation we introduce the functions
\begin{equation}\label{eq:xi_p}
\begin{split}
	\xi_p& :=
	a\!\!\sup_{s\in [p,\infty)}
	\max\big\{|\kp_1(s)|,|\kp_2(s)|\big\}
	=
	a\!\!\sup_{s\in [p,\infty)}
	\max\left\{|\gg(s)|, \frac{|\df(\phi(s))|\dph(s)}{\phi(s)}\right\},\\
	\zeta_p &:=\left (\frac{1-\xi_p}{1+\xi_p}\right )^2.
\end{split}
\end{equation}
In view of
the condition $a\|\kp_j\|_\infty < 1$ for $j= 1,2$, Lemma~\ref{lem:phi_bnd} and Proposition~\ref{prop:gamma}\,(i) yield that
the functions $\cdR_+\ni p\mapsto\xi_p,\zeta_p$
satisfy
$0 <\xi_p < 1$, $\lim_{p\arr\infty}\xi_p = 0$, and $\lim_{p\arr\infty}\zeta_p =1$.

\subsection{Spectral properties of auxiliary one-dimensional Schr\"odinger operators}\label{ssec:1D}
%
Consider the one-dimensional Schr\"odinger operator
\begin{equation}\label{eq:SL}
	\sfh_q\psi = -\psi'' + q\psi, \qquad \dom\sfh_q = H^2(\dR_+)\cap H^1_0(\dR_+),
\end{equation}
acting in the Hilbert space $L^2(\dR_+)$ with a real-valued potential $q\in L^\infty(\dR_+)$.
The operator $\sfh_q$ is obviously self-adjoint, and
the corresponding quadratic form will be denoted by $\frh_q$. The next proposition describes relations between the
spectral properties of $\sfh_q$ and the behavior of the potential
$q(s)$ as $s\arr\infty$.
It covers only the classes of potentials required for the proof of our main result.
\begin{prop}\label{prop:SL}
	Let the self-adjoint operator $\sfh_q$ be as in~\eqref{eq:SL} with the potential
	$q\colon\dR_+\arr\dR$ satisfying
	\begin{equation}\label{eq:qq}
		 |q(s)| \le
		 \frac{C}{(1+s)^\beta}
		 \and
		 \lims s^\beta q(s) = -c,
	\end{equation}
	with $\beta \in (0,2]$,  $C > 0$, and $c\ge 0$.
	Then the following claims hold.
	\begin{myenum}
		\item $\sess(\sfh_q) = [0,\infty)$.
		\item For $\beta = 2$ and $c = 0$, it holds $\cN_0(\sfh_q) < \infty$.
		\item For $\beta\in (0,2)$ and $c > 0$, it holds $\cN_0(\sfh_q) = \infty$
		and
		\[
			\cN_{-E}(\sfh_q)
			\underset{E\searrow 0}{\sim}
			\frac{c^\frac{1}{\beta} \sfB\left (
				\frac32, \frac{1}{\beta} - \frac{1}{2}  \right ) }
				{\pi\beta E^{\frac{1}{\beta} -\frac12} },
		\]
		where $\sfB(\cdot,\cdot)$ is the Euler beta-function.
		\end{myenum}		
\end{prop}
\begin{proof}
	The claim~(i) follows from~\cite[Lem. 9.35]{Te}, because the potential $q$ satisfies the condition
	\[
		\int_n^{n+1} |q(s)|\, \dd s
		\le
		\int_n^{n+1} \frac{C\,\dd s}{(1+s)^\beta}
	   \le
		\frac{C}{(1+n)^\beta}\arr 0,\qquad n\arr\infty.
	\]
	The claims~(ii) and~(iii) follow
	from~\cite[Lem 4.9]{BKRS09},
	see also~\cite[Thm. XIII.9 (a)]{RS4}
	and~\cite[Thm. XIII.82]{RS4}.
	We remark that the original claim of~\cite[Lem 4.9]{BKRS09} is formulated for operators
	on the full line, but the result for half-line operators follows easily with an additional factor $\frac12$ in the asymptotic formula for the eigenvalue counting function.
\end{proof}

\section{Bracketing lemma}
Pick an arbitrary $p > 0$ and define the line segment $\cS\subset\cG$ by
\[
	\cS :=\tau(\{p\}\tm \dI_a),
\]
where the mapping $\tau$ is as in~\eqref{eq:tau_k}.
For the sake of brevity, we mostly avoid indicating the dependence on $p$ as long as there is no danger of misunderstanding.
We define the open sets $\Omg_0 := (0,p)\tm \dI_a$ and $\Omg_1 := (p,\infty)\tm \dI_a$.
Clearly, the line segment $\cS$ splits
the parabolic strip $\cG$ into two
disjoint sub-domains
$\cG_j := \tau(\Omg_j)$, $j=0,1$.
In particular, we have the orthogonal decompositions
\[
	L^2(\cG) = 	L^2(\cG_0)\oplus L^2(\cG_1)
	\and
	L^2(\cG;r\dd r \dd z)
	=
	L^2(\cG_0;r\dd r\dd z)\oplus L^2(\cG_1;r\dd r\dd z).
\]
Next, we introduce for $m\in\dZ$ the following symmetric, densely defined quadratic forms in $L^2(\cG_1)$,
\begin{subequations}\label{eq:frf1}
\begin{align}
	\frf_{m1}^{\rm N}[\psi]  & := \int_{\cG_1}\cF_m[\psi]\,\dd r\dd z,	
	& \dom \frf_{m1}^{\rm N} & := \{ \Psi|_{\cG_1} \colon \Psi\in H^1_0(\cG)\},	
	\label{eq:frf1N}\\
	\frf_{m1}^{\rm D}[\psi]  & := \int_{\cG_1}\cF_m[\psi]\,\dd r \dd z ,	
	& \dom \frf_{m1}^{\rm D} & := H^1_0(\cG_1),\label{eq:frf1D}
\end{align}
\end{subequations}
where the integrand $\cF_m[\psi]$ is as in~\eqref{eq:EF}.
It is easy to see that both the forms $\frf_{m1}^{\rm N}$ and $\frf_{m1}^{\rm D}$
are closed and semi-bounded. Now we are in position to formulate the bracketing lemma which we employ in the following.
\begin{lem}\label{lem:bra}
	For any $p > 0$, $m\in\dZ$, and all $E > 0$ the following claims hold.
	\begin{myenum}
		\item $\cNE(\sF_m) \ge \cNE(\frf_{m1}^{\rm D})$.
		\item $\cNE(\sF_m) \le \cNE(\frf_{m1}^{\rm N}) + C^{\rm N}_m\,$
		for some constant $C^{\rm N}_m = C^{\rm N}_m(p) > 0$.
	\end{myenum}	
\end{lem}
\begin{proof}
	{\rm (i)} First, we define the restriction of the quadratic form
	$\sF_m$, $m\in\dZ$, to functions vanishing on $\cS$,
	\begin{equation}\label{eq:form_0_+}
		\sF_m^\cS[\psi] := \sF_m[\psi],\qquad
		\dom \sF_m^\cS  := \big\{\psi\in \dom \sF_m \colon \psi|_\cS = 0\big\}.
	\end{equation}
	The quadratic form $\sF_m^\cS$ can be naturally decomposed into the orthogonal sum
	$\sF_m^\cS = \sF_{m0}^\cS\oplus \sF_{m1}^\cS$, where the quadratic forms $\sF_{m0}^\cS$ and $\sF_{m1}^\cS$ are defined in the Hilbert spaces $L^2(\cG_0;r\dd r \dd z)$ and $L^2(\cG_1;r\dd r \dd z)$, respectively.
	Using the ordering $\sF_m \prec \sF^\cS_m$
	of the forms and the min-max principle we obtain
	\[
		\cNE(\sF_m) \ge \cNE(\sF_m^\cS) = \cNE(\sF_{m0}^\cS) +  \cNE(\sF_{m1}^\cS),	
		\qquad \forall\, E > 0.
	\]
	It remains to notice that the form $\sF_{m1}^\cS$ is unitarily equivalent
	via
	\begin{equation}\label{eq:U1}
		\sfU_1\colon L^2(\cG_1;r\dd r\dd z)\arr L^2(\cG_1), \qquad \sfU_1\psi := \sqrt{r}\psi,
	\end{equation}
	to the form $\frf_{m1}^{\rm D}$ in~\eqref{eq:frf1D}, and therefore the desired inequality in~{\rm (i)} holds.
	\smallskip
	
   \noindent {\rm (ii)}
	Let $\chi_0\colon \cdR_+\arr [0,1]$
	be a $\cC^\infty$-smooth function satisfying $\chi_0(s) = 1$ for all $s < \frac12$
	and $\chi_0(s) = 0$ for all $s > 1$.
	Define the function $\chi_1\colon \cdR_+\arr [0,1]$ through the identity $\chi_0^2 + \chi_1^2 \equiv  1$
	on $\cdR_+$.	Employing the curvilinear coordinates $(s,u)$ on the parabolic strip $\cG$
	we introduce the cut-off functions $\cX_{j,p}\colon \cG\arr [0,1]$, $j=0,1$, by
	\[
		\cX_{0,p}(s, u) := \chi_0(p^{-1} s)	\and	\cX_{1,p}(s, u) := \chi_1(p^{-1}s),
		\qquad p >0.
	\]
	These cut-off functions can also be viewed as functions of the arguments $(r,z)$.
	Define the expression $X_p \colon\cG\arr\dR_+$ by
	\begin{equation}\label{eq:Weps}
		X_p := |\nb\cX_{0,p}|^2 + |\nb\cX_{1,p}|^2.
	\end{equation}
	Set $\Omg_0' := (0,p)\tm \dI_a$ and $\Omg_1' := (\frac{p}{2},\infty)\tm\dI_a$.
	The parabolic strip $\cG$ can be represented as the union of
	the domains $\cG_0' := \tau(\Omg_0')$ and
	$\cG_1' := \tau(\Omg_1')$, having a non-empty
	bounded intersection $\cG_{01}' := \cG_0'\cap\cG_1'$.
	The expression $X_p$ in~\eqref{eq:Weps} can be pointwise estimated from above as $X_p \le C_\cX p^{-2}\one_{\cG_{01}'}$
	with some constant $C_\cX > 0$, here $\one_{\cG_{01}'}$ is the characteristic function of $\cG_{01}'$.
	Next we define a family of quadratic forms parametrized by $m\in\dZ$ and $j=0,1$ through
	\[
	\begin{split}
		\sE_{mj}^{\rm D}[\psi] 	& :=
		\int_{\cG_j'}\cE_m[\psi] r \dd r \dd z - C_\cX p^{-2}\int_{\cG_{01}'}|\psi|^2 r \dd r \dd z,\\
		\dom \sE_{mj}^{\rm D}   & :=
		\big\{\Psi|_{\cG_j'}\colon \Psi \in \dom\sF_m, {\rm supp}\,\Psi\subset\ov{\cG_j'}\big\},
	\end{split}
	\]
	where the integrand $\cE_m[\psi]$ is as in~\eqref{eq:EF}. Further,
	applying the IMS formula~\cite[\S 3.1]{CFKS} to the quadratic form $\sF_m$, we obtain that 
	$(\cX_{j,p}\psi)|_{\cG_j'} \in \dom\sE_{mj}^{\rm D}$, $j=0,1$ and with a slight
	abuse of notation we have
	\[
		\sF_m [\psi] \ge
		\sE_{m0}^{\rm D}[\cX_{0,p}\psi]
			+ \sE_{m1}^{\rm D}[\cX_{1,p}\psi],
		\qquad \forall\,\psi\in\dom\sF_m.	
	\]
	Hence, by~\cite[Lem. 5.2]{DLR12} we get
	\[
		\cNE(\sF_m) \le \cNE(\sE_{m0}^{\rm D}) + \cNE(\sE_{m1}^{\rm D}), \qquad\forall\,E >0.
	\]
	Notice that the quadratic form $\sE_{m0}^{\rm D}$ corresponds to a self-adjoint operator with a compact resolvent, while the form $\sE_{m1}^{\rm D}$ is unitarily equivalent via $\sfU_1'\colon L^2(\cG_1';r\dd r \dd z)\arr L^2(\cG_1')$, $\sfU_1'\psi := \sqrt{r}\psi$, to the following form
	in the Hilbert space $L^2(\cG_1')$,
	\[
		H^1_0(\cG_1')\ni \psi
		\mapsto 	
		\int_{\cG_1'}\cF_m[\psi]\dd r \dd z - C_\cX p^{-2}\int_{\cG_{01}'} |\psi|^2 \dd r\dd z,
	\]
	which is larger in the sense of ordering than the orthogonal sum $\frf\oplus\frf_{m1}^{\rm N}$ of the form
	$\frf_{m1}^{\rm N}$ in the Hilbert space $L^2(\cG_1)$
	defined in~\eqref{eq:frf1N} and the form in the Hilbert space $L^2(\cG_{01}')$,
	\[
		\frf[\psi] :=		
		\int_{\cG_{01}'}\big(\cF_m[\psi] - C_\cX p^{-2}|\psi|^2\big) \dd r\dd z
		,\qquad 
		\dom \frf :=
		\big\{\Psi|_{\cG_{01}'}\colon \Psi \in H^1_0(\cG_1')\big\}.
	\]
	It only remains to notice that
	the last form also corresponds to a
	self-adjoint operator in the Hilbert space $L^2(\cG_{01}')$
	with a compact resolvent. Thus, the desired
	inequality holds with the constant $C_m^{\rm N} =
	\cN_\lme(\sE_{m0}^{\rm D}) + \cN_\lme(\frf)$.
\end{proof}

\section{Straightening of the meridian domain}
\label{sec:stra}

Recall that the curve $\G$ in~\eqref{eq:Gamma}
is parametrized via the mapping $\cdR_+\ni s\mapsto (\phi(s), f(\phi(s)))$
with $\phi$ satisfying the differential equation~\eqref{eq:ODE_phi}.
Furthermore, recall that the quadratic form $\sF_m$ in~\eqref{eq:forms}
is unitarily equivalent
through the unitary transformation $\sfU$ to the form $\frf_m$ in~\eqref{eqn:fq_flat}
expressed in the flat metric.

Let us introduce auxiliary potentials
$V_m^{\rm C}, V_m\colon\Omg\arr \dR$, $m\in\dZ$,
on the half-strip $\Omg = \dR_+\tm\dI_a$ by the formul\ae
\begin{equation}\label{eq:Vm}
\begin{split}
	V_m^{\rm C}(s,u)
	&:=
	\frac{m^2-\frac14}{\big(\phi(s) - u\df(\phi(s))\dph(s)\big)^2},
	\qquad \\
	V_m(s,u)
	& := \frac{u\ddot\gg(s)}{2 g^{3/2}(s,u)}
	   - \frac{\gg^2(s)}{4 g(s,u)}
   	- \frac54\frac{u^2\dot\gg^2(s)}{g^2(s,u)}
		+ V_m^{\rm C}(s,u).
\end{split}	
\end{equation}
The first three terms in the definition of $V_m$
correspond to the curvature-induced potential arising while straightening
the curved two-dimensional half-strip $\cG$; \cf~\cite[Sec. 1.1]{EK}. The fourth term
$V^{\rm C}_m$ corresponds
to the centrifugal potential written in the $(s,u)$-coordinates.

For the sake of brevity, we introduce shorthand notations
for the integrands,
\begin{equation}
\begin{aligned}[b]
		\cT_m[\psi](s,u)
		& :=
		\frac{|\p_s\psi|^2}{g(s,u)} + |\p_u\psi|^2 + V_m|\psi|^2,\qquad &m&\in\dZ,\\
		\cB_p[\psi](u) & := \frac{\gg'(p) u|\psi(p,u)|^2}{2 g^{3/2}(p,u)},\qquad &p & \ge 0.
\end{aligned}
\end{equation}
Next, we define the unitary operator $\sfV\colon L^2(\cG)\arr L^2(\Omg)$ by
\[
	(\sfV\psi)(s,u) := g(s,u)^{1/4}\,\psi\big(\tau(s,u)\big),
\]
where the map $\tau$ is as in~\eqref{eq:tau_k},
its Jacobian $\cJ$ is given by~\eqref{eq:cJ}
and $g = \cJ^2$ as in~\eqref{eq:gk}.
The quadratic form $\frf_m$ is unitarily equivalent via
the mapping $\sfV$ to the quadratic form
\begin{equation}\label{eq:tm1}
	\frt_m[\psi] := \frf_m[\sfV^{-1}\psi],
	\qquad	
	\dom\frt_m := \sfV(\dom\frf_m),
\end{equation}
on the Hilbert space $L^2(\Omg)$.
Furthermore, recall that $\Omg_1 = (p,\infty)\tm\dI_a\subset\Omg$ and introduce the unitary operator $\sfV_1\colon L^2(\cG_1)\arr L^2(\Omg_1)$ by
\[
	(\sfV_1\psi)(s,u) := g(s,u)^{1/4}\,\psi\big(\tau(s,u)\big).
\]
The quadratic forms $\frf_{m1}^{\rm D}$
and $\frf_{m1}^{\rm N}$ in~\eqref{eq:frf1} can be further unitarily transformed into the forms
\begin{align}
	\frt_{m1}^{\rm D}[\psi] & := \frf_{m1}^{\rm D}[\sfV^{-1}_1\psi],
	&\dom\frt_{m1}^{\rm D}  & := \sfV_1(\dom\frf_{m1}^{\rm D}),
	\label{eq:tm1D}\\
	\frt_{m1}^{\rm N}[\psi] & := \frf_{m1}^{\rm N}[\sfV^{-1}_1\psi],
	& \dom\frt_{m1}^{\rm N}&  := \sfV_1(\dom\frf_{m1}^{\rm N}), 	\label{eq:tm1N}
\end{align}	
on the Hilbert space $L^2(\Omg_1)$.

In the remaining part of this section
we will get more explicit expressions for the forms
$\frt_m$, $\frt_{m1}^{\rm D}$ and $\frt_{m1}^{\rm N}$. 
Let $\Psi \in H^1(\cG)$ and 
$\psi \in H^1(\Omg)$ be connected through
the relation $\Psi \circ \tau = \cJ^{-1/2}\psi$.
Let also $(t_1,t_2)^\top$ and
$(n_1,n_2)^\top$ 
be, respectively, the unit tangential and  normal vectors to $\G$.
Using the Frenet formula we find
\[
\begin{split}
	\p_s(\cJ^{-1/2}\psi)
	&= 
	(\p_r\Psi\circ\tau)\p_s \tau_1
	+
	(\p_z\Psi\circ\tau)\p_s \tau_2\\
	& =
	(\p_r\Psi\circ\tau)t_1(1 - u\gg) 
	+
	(\p_z\Psi\circ\tau)t_2(1 - u\gg)
	=\cJ\left(
	(\p_r\Psi\circ\tau)t_1 
	+
	(\p_z\Psi\circ\tau)t_2\right),\\[0.5ex]
	\p_u(\cJ^{-1/2}\psi) & =
	(\p_r\Psi\circ\tau)\p_u \tau_1
	+
	(\p_z\Psi\circ\tau)\p_u \tau_2
	=
	(\p_r\Psi\circ\tau) n_1
	+
	(\p_z\Psi\circ\tau)n_2,\\
\end{split}
\]
Hence, we obtain
\begin{equation}\label{eq:nablaPsi}
	|(\nabla \Psi)\circ\tau|^2 =
	\frac{|\p_s(\cJ^{-1/2}\psi)|^2}{\cJ^2} + |\p_u(\cJ^{-1/2}\psi)|^2.
\end{equation}
In the spirit of~\cite[Sec. 1.1]{EK}, we find
that $\dom \frt_m = H^1_0(\Omg)$
for all $m\in\dZ\sm\{0\}$
and by performing elementary computations
relying on~\eqref{eq:nablaPsi} we get
\begin{equation*}
\begin{split}
	\frt_m[\psi]
	& =
	\int_{\Omg}
	\frac{|\p_s(\cJ^{-1/2}\psi)|^2}{\cJ} + |\p_u(\cJ^{-1/2}\psi)|^2\cJ +
	V_m^{\rm C}|\cJ^{-1/2}\psi|^2\cJ\\
	& =
	\int_\Omg
	\frac{|\p_s\psi|^2}{\cJ^2} -
	\frac{u\dot\gg\Re(\ov{\psi} \p_s\psi)}
		{\cJ^3} +
	\frac{u^2\dot\gg^2|\psi|^2}{4\cJ^4}
	+ |\p_u\psi|^2 - \frac{\gg
	\Re(\ov{\psi}\p_u\psi)}{\cJ} +
	\frac{\gg^2|\psi|^2}{4\cJ^{2}}
	+ 	V_m^{\rm C}|\psi|^2
	\\
	& =
	\int_\Omg
	\frac{|\p_s\psi|^2}{\cJ^2} -
	\frac{u\dot\gg\p_s(|\psi|^2)}
	{2\cJ^3} +
	\frac{u^2\dot\gg^2|\psi|^2}{4\cJ^4}
	+ |\p_u\psi|^2 - \frac{\gg
		\p_u(|\psi|^2)}{2\cJ} +
	\frac{\gg^2|\psi|^2}{4\cJ^{2}} +
	V_m^{\rm C}|\psi|^2.
\end{split}
\end{equation*}
Integrating by parts in the above formula,
we obtain
\begin{equation}\label{eq:frtm}
\begin{split}
	\frt_m[\psi] & =
	\int_\Omg
	\frac{|\p_s\psi|^2}{\cJ^2}
	+
	|\p_u\psi|^2 +
	\left[
	\frac{u^2\dot\gg^2}{4\cJ^4}
	+
	\p_s\left(\frac{u\dot\gg}
	{2\cJ^3}\right) +
	\p_u\left(\frac{\gg}{2\cJ}\right)
	+
	\frac{\gg^2}{4\cJ^{2}} + V_m^{\rm C}\right] |\psi|^2
	\\
	& =
	\int_\Omg
	\frac{|\p_s\psi|^2}{\cJ^2}
	+
	|\p_u\psi|^2 +
	\left[
	\frac{u^2\dot\gg^2}{4\cJ^4}
	+
	\frac{u\ddot\gg}{2\cJ^3}
	-
	\frac{3u^2\dot\gg^2}{2\cJ^4}
	 -
	\frac{\gg^2}{2\cJ^2}
	+
	\frac{\gg^2}{4\cJ^{2}} + V_m^{\rm C}
	\right] |\psi|^2
	\\
	& =
	\int_\Omg
	\frac{|\p_s\psi|^2}{\cJ^2}
	+
	|\p_u\psi|^2 +
	\left[
	\frac{u\ddot\gg}{2\cJ^3}
	-
	\frac{5u^2\dot\gg^2}{4\cJ^4}
	-
	\frac{\gg^2}{4\cJ^2} + V_m^{\rm C}	\right] |\psi|^2
	\\
	& = \int_\Omg\cT_m[\psi],
\end{split}	
\end{equation}
the boundary terms vanished thanks to the Dirichlet boundary condition.

Analogously, we get for all $m\in\dZ$
%
\[
	\frt_{m1}^{\rm D}[\psi]
	:=
	\int_{\Omg_1}\cT_m[\psi]\,\dd s \dd u,
	\qquad \dom \frt_{m1}^{\rm D} := H^1_0(\Omg_1),
\]
Mimicking the above
computation for the form $\frt_{m1}^{\rm N}$, $m\in\dZ$, we arrive at  $\dom \frt_{m1}^{\rm N} = 
\{\psi|_{\Omg_1}\colon \psi \in H^1_0(\Omg)\}$ and obtain that		 
\[
	\frt_{m1}^{\rm N}[\psi]
	 :=
	\int_{\Omg_1}
	\frac{|\p_s\psi|^2}{\cJ^2} -
	\frac{u\dot\gg\p_s(|\psi|^2)}
	{2\cJ^3} +
	\frac{u^2\dot\gg^2|\psi|^2}{4\cJ^4}
	+ |\p_u\psi|^2 - \frac{\gg
		\p_u(|\psi|^2)}{2\cJ} +
	\frac{\gg^2|\psi|^2}{4\cJ^{2}} + V_m^{\rm C}|\psi|^2.
\]
Again integrating by parts, we finally get
\[
\begin{split}
	\frt_{m1}^{\rm N}[\psi]
	&  =
	 \int_{\Omg_1}
	 \left(
	 \frac{|\p_s\psi|^2}{\cJ^2}
	 +
	 |\p_u\psi|^2 +
	 V_m|\psi|^2\right)\dd s \dd u
	 +
	 \int_{\dI_a}
	 \frac{\dot\gg(p)u|\psi(p,u)|^2}
	 {2\cJ^3(p,u)}\dd u\\
	&=
	\int_{\Omg_1}\cT_m[\psi]\dd s\dd u
	+
	\int_{\dI_a}\cB_p[\psi]\dd u. 	
\end{split}
\]
While the Dirichlet boundary condition
on $\p\cG_1\sm\cS$
is preserved upon straightening of $\cG_1$,
the Neumann boundary condition on the line
segment $\cS$ transforms into the Robin
boundary condition with the coupling function
$\dI_a \ni u \mapsto
\frac{\dot\gg(p) u}{2\cJ^3(p,u)}$. This peculiarity manifests in the presence
of the boundary in the expression for
$\frt_{m1}^{\rm N}$.

\section{Properties of \boldmath{$\cNE(\sfF_m)$} for $m \in \dZ\sm\{0\}$}
\label{sec:nonrad_fib}
In this section, we investigate the spectral counting function of the fiber operators $\sfF_m$ for all $m \in \dZ \sm \{0\}$. First, we show that for all $m \ne 0$ the discrete spectrum of the fiber operators below the threshold $\lme$ is at most finite, and secondly, that this discrete spectrum is empty for all but finitely many such fiber operators. This informal explanation is precisely formulated in the proposition below, the proof of which relies on the min-max principle and on the finiteness of the discrete spectrum for a class of one-dimensional
Schr\"{o}dinger operators stated in Proposition~\ref{prop:SL}\,(ii).
\begin{prop}\label{prop:nonrad_fib}
	Let $\sfF_m$, $m\in\dZ\sm\{0\}$, be the self-adjoint fiber operators
	in the Hilbert space $L^2(\cG;r\dd r\dd z)$	
	associated with the quadratic forms in~\eqref{eq:forms}.
	Then there exists an $M = M(f,a)\in\dN_0$ such that:
	\begin{myenum}
		\item $1 \le \cN_\lme( \sfF_m ) < \infty$ for all $m = 1,2,\dots, M$;
		\item $\cN_\lme( \sfF_m ) = 0$ for all $m > M$.
	\end{myenum}	
\end{prop}	
\begin{proof}
	First of all, recall that the form $\frt_m$ in~\eqref{eq:tm1} is unitarily equivalent to the form $\sF_m$ which is represented by the operator $\sfF_m$.
	We organize the argument into three steps.
	\smallskip

	\noindent {\it Step 1:~bracketing}.
	%
	We infer that the ordering $\frs_m\prec \frt_m$ holds,
	where the form $\frs_m$ is given by
	\begin{equation}\label{eq:frs_m}
		\frs_m[\psi]
		 =
		\int_\Omg
		\left (\frac{|\p_s \psi|^2}{(1+\xi_0)^2} + |\p_u \psi|^2 + U_m|\psi|^2\right )\dd u \dd s,\qquad
		\dom \frs_m = H^1_0(\Omg),	
	\end{equation}
	with $\xi_0$ as in \eqref{eq:xi_p} for $p=0$ and the potential $U_m$ defined by
	\begin{equation}\label{eq:Um}
		U_m(s) =
		\frac{1}{\phi^2(s)}
		\left (
			 \frac{m^2-\frac14}{(1+\xi_0)^2} - \frac{\phi^2(s)\gg^2(s)}{4 (1 - \xi_0 )^2}
			 -
			 \frac{a\phi^2(s)|\ddot\gg(s)|}{2(1-  \xi_0)^3}
			 -
			 \frac{5}{4}\frac{a^2\phi^2(s)\dot\gg^2(s)}{(1-\xi_0)^4}
		\right ).
	\end{equation}
	\smallskip
	
	\noindent {\it Step 2:~large $|m|\in\dN$.}
	In view of Propositions~\ref{prop:phi_bnd}\,(i) and~\ref{prop:gamma}\,(i)-(iii),
	the functions
	\[
		s\mapsto \phi^2(s)\gg^2(s),\qquad s\mapsto \phi^2(s)|\ddot\gg(s)|\and	s\mapsto \phi^2(s)\dot\gg^2(s)
	\]
	are all bounded on $\cdR_+$.
	Hence, for any $|m| \ge M$ with $M\in\dN$ large enough,
	the potential $U_m$ is pointwise positive. Fur such $m$'s we obtain
	\[
		\cN_\lme(\sfF_m) = \cN_\lme(\frf_m) = \cN_\lme(\frt_m) \le \cN_\lme(\frs_m) = 0.
	\]	
	\smallskip
	\noindent {\it Step 3:~small $|m|\in\dN$.}
	In view of the asymptotics of $\gg(s)$, $\dot\gg(s)$
	and $\ddot\gg(s)$ as $s\arr\infty$ shown in
   Proposition~\ref{prop:gamma}, the potential $q\colon \dR_+\arr \dR$ defined by
	\begin{equation}\label{eq:q}
		q(s)
		:=
		-\frac{1}{\zeta_0}
		\left (
			\frac{\gg^2(s)}{4}
			+
			\frac{a|\ddot\gg(s)|}{2(1-\xi_0)}
			+
			\frac{5}{4}\frac{a^2\dot\gg^2(s)}{(1-\xi_0)^2}	
		\right ),
	\end{equation}
	where $\xi_0$ and $\zeta_0$ are as in~\eqref{eq:xi_p} with $p = 0$,
	satisfies the condition~\eqref{eq:qq}
	with $\beta = 2$ and $c = 0$
	and therefore by Proposition~\ref{prop:SL} the negative discrete spectrum
	of the  self-adjoint operator
	$\sfh_q$ in~\eqref{eq:SL} is finite.
	
	Next, we define the auxiliary quadratic form
	\begin{equation}\label{eq:form_Dir}
		H^1_0(\dI_a) \ni\psi \mapsto \frh^{\rm D}_a[\psi] :=  \|\psi'\|^2_{L^2(\dI_a)}
	\end{equation}
	in the Hilbert space $L^2(\dI_a)$
	and consider the closed, densely defined, symmetric
	and semi-bounded  quadratic form
	in $L^2(\Omg) = L^2(\dR_+)\otimes L^2(\dI_a)$  having the tensor product structure
	$\frh :=	(1+\xi_0)^{-2}\frh_q\otimes \fri + \fri\otimes\frh^{\rm D}_a$,
	where $\frh_q$ is the form represented
	by $\sfh_q$ in~\eqref{eq:SL} and  $\fri$ stands for the quadratic form of the identity operator on a generic Hilbert space.
	Comparing the definition~\eqref{eq:frs_m} of the form $\frs_m$,
	the expression~\eqref{eq:Um} for the potential $U_m$, and the expression~\eqref{eq:q}
	for the potential $q$, we obtain that the ordering $\frh\prec \frs_m$ holds for all $m \ne 0$.	
	Furthermore, let $\lm_n^{\rm D}(a)$, $n\in\dN$, be the eigenvalues
	of $\frh^{\rm D}_a$ enumerated in a non-decreasing way.
	Note also that we have $\lme = \lm_1^{\rm D}$.
	Since $\frh_q$ is semi-bounded, there exists $N = N(f,a)$
	such that $(1+\xi_0)^{-2}\frh_q \ge \lme - \lm_N^{\rm D}$.
	Consequently, for any $|m| < M$ we get
	\[
	\begin{split}
		\cN_\lme(\sfF_m) & = \cN_\lme(\frf_m) =
		\cN_\lme(\frt_m) \le \cN_\lme(\frs_m)\\
		& \le
		\sum_{n=1}^\infty\cN_{\lme - \lm_n^{\rm D}}\big((1+\xi_0)^{-2}\frh_q\big)
		\le
		\sum_{n=1}^N\cN_0\big((1+\xi_0)^{-2}\frh_q\big)
		= N\cdot \cN_0(\frh_q) < \infty,
	\end{split}	
	\]
which concludes the proof.
\end{proof}

\section{Asymptotics of \boldmath{$\cNE(\sfF_0)$}}
The fiber operator $\sfF_0$ corresponding to $m = 0$
requires a separate consideration. Recall that this
fiber operator is associated with the quadratic form
$\sF_0$ on the Hilbert space $L^2(\cG;r\dd r\dd z)$ defined
as
\begin{equation}\label{eq:form0}
	\sF_0[\psi] = \int_\cG \cE_0[\psi]r\dd r\dd z,
	\qquad
	\dom \sF_0 = \Pi_0(H^1_0(\cP)),
\end{equation}
where $\cE_0$ is as in~\eqref{eq:EF}.
Throughout this section we use the shorthand notation
\begin{equation}\label{eq:g}
	\sfg_{\aa,k}(E) :=
	\frac{1}{2\pi}
	\frac{\aa k}{2^\aa}
	\frac{\sfB\left (\frac32,\frac{\aa}{2}- \frac12  \right ) }{E^{\frac{\aa}{2} -\frac12}}.
\end{equation}

\subsection{A lower bound on \boldmath{$\cNE(\sfF_0)$}}
In this subsection we obtain an asymptotic lower bound on the counting function
$\cNE(\sfF_0)$ in the limit $E\arr 0^+$. To this aim we modify the strategy
used in the proof of Proposition~\ref{prop:nonrad_fib}  by additionally involving Lemma~\ref{lem:bra}\,(i).
At the end we reduce the problem to the spectral asymptotics of a one-dimensional Schr\"odinger operator covered by Proposition~\ref{prop:SL}\,(iii).
\begin{prop}\label{prop:lower_bnd}
	Let the self-adjoint fiber operator $\sfF_0$
	in the Hilbert space $L^2(\cG;r\dd r\dd z)$
	be associated with the form~\eqref{eq:form0}
	and let the function $\sfg_{\aa,k}(\cdot)$ be as in~\eqref{eq:g}. 	
	Then
	\[
		\liminf_{E\arr 0^+}\frac{\cN_{\lme - E}(\sfF_0)}{\sfg_{\aa,k}(E)} \ge 1.
	\]
\end{prop}
\begin{proof}
	We reduce finding the asymptotic lower bound	to the analysis of one-dimensional Schr\"odinger operators.
	Recall first that Lemma~\ref{lem:bra}\,(i)
	and the construction described in Section~\ref{sec:stra}
	imply that
	\[
		\cNE(\sfF_0) \ge \cNE(\frt_{01}^{\rm D}),	\qquad \forall\, E >0.
	\]
	Let the auxiliary functions $\xi_p$ and $\zeta_p$ be as in~\eqref{eq:xi_p}.
	Next we introduce the quadratic form
	\begin{equation}\label{eq:fru}
		\fru_{01}^{\rm D}[\psi]	= \int_{\Omg_1}
		\left (
			\frac{|\p_s\psi|^2}{(1 - \xi_p)^2}	+ |\p_u\psi|^2 + W_p(s)|\psi|^2
		\right )\dd u \dd s,
		\qquad
		\dom \fru_{01}^{\rm D}   = H^1_0(\Omg_1),	
	\end{equation}
	where $\Omg_1 = (p,\infty)\tm \dI_a$ and the potential $W_p(s)$ is given by
	\begin{equation}\label{eq:Wp}
		W_p(s) =
		\frac{a|\ddot\gg(s)|}{2(1- \xi_p)^3}
		-
		\frac{1}{4(1+\xi_p)^2}\frac{1}{\phi^2(s)}.
	\end{equation}
	For $g$ in~\eqref{eq:gk} and $V_{0}$ in~\eqref{eq:Vm} with $m = 0$,
	we have $(1-\xi_p)^2\le g(s,u) \le (1+\xi_p)^2$ and $V_{0}(s,u) \le W_p(s)$ for all $(s,u)\in \Omg_1$,
	hence the ordering of the forms
	$\frt_{01}^{\rm D}\prec \fru_{01}^{\rm D}$ holds.

	Now we can reduce the asymptotics analysis
	to investigation of one-dimensional Schr\"odinger operators.
	Define the potential $q \colon \dR_+\arr \dR$ as
	\begin{equation}\label{eq:q1}
		q(s) := W_p(s-p)\,(1-\xi_p)^2.
	\end{equation}
	Proposition~\ref{prop:phi_bnd}\,(i) and
	Proposition~\ref{prop:gamma}\,(ii) imply that the potential $q$ in~\eqref{eq:q1} satisfies the condition~\eqref{eq:qq} in Proposition~\ref{prop:SL}
	with $\beta = \frac{2}{\aa} \in (0,2)$ and $c =  \frac{\zeta_p}{4} k^{\frac{2}{\aa}}$.
	Consequently, using the decomposition
	\[
		\fru_{01}^{\rm D} = (1-\xi_p)^{-2}\frh_q\otimes \fri + \fri\otimes \frh^{\rm D}_a
		\cong
		\bigoplus_{n\in\dN}
		\left( \lm_n^{\rm D}(a) + (1-\xi_p)^{-2}\frh_q\right )
	\]
	with $\frh^{\rm D}_a$ as in~\ref{eq:form_Dir},
	and applying Proposition~\ref{prop:SL}\,(iii),
	we obtain that
	\[
	\begin{split}
		\liminf_{E\arr 0^+}
		\frac{\cNE(\sfF_0)}{\sfg_{\aa,k}(E)}
		&
		\ge
		\liminf_{E\arr 0^+}
		\frac{\cNE(\frt_{01}^{\rm D})}{\sfg_{\aa,k}(E)}\\
		& \ge
		\liminf_{E\arr 0^+}
		\frac{\cNE(\fru_{01}^{\rm D})}{\sfg_{\aa,k}(E)}
		\ge
		\liminf_{E\arr 0^+}
		\frac{\cN_{-E(1-\xi_p)^2}(\frh_q)}{\sfg_{\aa,k}(E)}\\
		& =
		\liminf_{E\arr 0^+}
		\frac{1}{2\pi}
		\frac{\aa k}{2^\aa}
		\frac{\sfB\left (\frac32,\frac{\aa}{2}- \frac12  \right ) }{E^{\frac{\aa}{2} -\frac12}}
		\frac{1}{\sfg_{\aa,k}(E)}
		\zeta_p^{\frac{\aa}{2}} (1-\xi_p)^{1 -\aa}
		 =
		\zeta_p^{\frac{\aa}{2}} (1-\xi_p)^{1 -\aa}.
	\end{split}	
	\]
	Eventually, passing to the limit $p \arr \infty$ in the above
	inequality and making use of $\xi_p\arr 0$, $\zeta_p\arr 1$ as $p\arr\infty$ we obtain the claim.
\end{proof}

\subsection{An upper bound on \boldmath{$\cNE(\sfF_0)$}}
Finally, we derive an asymptotic upper bound on $\cNE(\sfF_0)$.
To this aim we combine the strategy
used in the proof of Proposition~\ref{prop:nonrad_fib} 
with Lemma~\ref{lem:bra}\,(ii).
\begin{prop}\label{prop:upper_bnd}
	Let the self-adjoint fiber operator $\sfF_0$
	in the Hilbert space $L^2(\cG;r\dd r\dd z)$
	be associated with the form~\eqref{eq:form0}
	and let the function $\sfg_{\aa,k}(\cdot)$ be as in~\eqref{eq:g}.
	Then
	\[
		\limsup_{E\arr 0^+}\frac{\cNE(\sfF_0)}{\sfg_{\aa,k}(E)} \le 1.
	\]
\end{prop}

\begin{proof}
	We reduce finding the asymptotic upper bound again
	to the analysis of one-dimensional Schr\"odinger operators.
	Recall that Lemma~\ref{lem:bra}\,(ii)
	and the construction described in Section~\ref{sec:stra}
	imply that
	\[
		\cNE(\sfF_0) \le \cNE(\frt_{01}^{\rm N}) + C_0^{\rm N},	\qquad \forall\, E >0.
	\]
	Let the auxiliary functions $\xi_p$ and $\zeta_p$ be as in~\eqref{eq:xi_p}.	
	Introduce the quadratic form
	\begin{equation}\label{eq:frsN}
	\begin{split}
		\frs_{01}^{\rm N}[\psi]
		& =
		\int_{\Omg_1}
		\left (\frac{|\p_s\psi|^2}{(1+\xi_p)^2} + |\p_u\psi|^2 + U_p(s) |\psi|^2\right )
		\dd s	\dd u - \int_{\dI_a}\frac{a|\dot\gg(p)||\psi(p,u)|^2}{2(1-\xi_p)^3}\dd u,\\
		\dom\frs_{01}^{\rm N} & :=
		\big\{\Psi|_{\Omg_1}\colon \Psi \in H^1_0(\Omg)\big\},
	\end{split}
	\end{equation}
	where $\Omg_1 = (p,\infty)\tm \dI_a$ and the potential $U_p$ is given by
	\begin{equation}\label{eq:Up}
		U_p(s)
		:=
		-\frac{\gg^2(s)}{4(1-\xi_p)^2}
		-
		\frac{5a^2\dot\gg^2(s)}{4(1-\xi_p)^4}
		-
		\frac{a|\ddot\gg(s)|}{2(1- \xi_p)^3}
  		-
		\frac{1}{4(1-\xi_p)^2}\frac{1}{\phi^2(s)}.
	\end{equation}
	For $g$ in~\eqref{eq:gk} and $V_{0}$ in~\eqref{eq:Vm} with $m = 0$,
	we have $(1-\xi_p)^2\le g(s,u) \le (1+\xi_p)^2$ and $V_0(s,u) \ge U_p(s)$ for all $(s,u)\in \Omg_1$,
	hence the ordering of the forms
	$\frs_{01}^{\rm N}\prec \frt_{01}^{\rm N}$ holds.
	
	Define the potential $q \colon \dR_+\arr \dR$ as	
	\begin{equation}\label{eq:q2}
		q(s) := U_p(s-p)(1 + \xi_p)^2.
	\end{equation}	
	Proposition~\ref{prop:phi_bnd}\,(i)
	and Proposition~\ref{prop:gamma}\,(i)-(iii)
	imply
	that $q$ satisfies the condition~\eqref{eq:qq}
	with $\beta = \frac{2}{\aa}\in (0,2)$ and
	$c =\frac{k^{\frac{2}{\aa}}}{4\zeta_p}$.
	Next we define the quadratic form
	\[
		H^1(\dR_+) \ni \psi
		\mapsto \frh_q'[\psi] :=
		\int_{\dR_+}\left (|\psi'(s)|^2 + q(s)|\psi(s)|^2\right)\dd s - \frac{a|\dot\gg(p)||\psi(0)|^2}{2(1-\xi_p)\zeta_p},
	\]
	noting that the self-adjoint operator $\sfh_q'$ corresponding to the quadratic form
	$\frh_q'$ is a rank-one perturbation of the operator $\sfh_q$ defined in~\eqref{eq:SL}
	with the potential as in~\eqref{eq:q2}. Hence by~\cite[\S 9.3, Thm. 3]{BS} we have
	\begin{equation}\label{}
		|\cN_{-E}(\frh_q') - \cN_{-E}(\frh_q)| \le 1,\qquad \forall\, E >0.
	\end{equation}
	In the tensor-product decomposition of $\frs_{01}^{\rm N}$
	into the orthogonal sum,
	\[
		\frs_{01}^{\rm N} = (1+\xi_p)^{-2}\frh_q'\otimes\fri + \fri\otimes \frh^{\rm D}_a \cong
		\bigoplus_{n\in\dN}\big( \lm_n^{\rm D}(a) + (1+\xi_p)^{-2}\sfh_q'\big),
	\]
	only finitely many summands have non-empty
	discrete spectrum below the
	threshold $\lm_1^{\rm D}(a) = \lme$
	and for all of them except for the lowest one
	the discrete spectrum below $\lme$ is finite.
	Making now use of Proposition~\ref{prop:SL}\,(iii) with $c = \frac{k^{\frac{2}{\aa}}}{4\zeta_p} $
	and $\beta = \frac{2}{\aa}$ we obtain
	\[
	\begin{split}
		\limsup_{E\arr 0^+}
		\frac{\cNE(\sfF_0)}{\sfg_{\aa,k}(E)}
		&
		\le
		\limsup_{E\arr 0^+}
		\frac{\cNE(\frt_{01}^{\rm N})}{\sfg_{\aa, k}(E)}
		 \le
		\limsup_{E\arr 0^+}
		\frac{\cNE(\frs_{01}^{\rm N})}{\sfg_{\aa, k}(E)}\\
		& \le
		\limsup_{E\arr 0^+}
		\frac{\cN_{-E(1+\xi_p)^2}(\sfh_q')}{\sfg_{\aa,k}(E)}
		=
		\limsup_{E\arr 0^+}
		\frac{\cN_{-E(1+\xi_p)^2}(\sfh_q)}{\sfg_{\aa,k}(E)}\\
		& =
		\limsup_{E\arr 0^+}
		\frac{1}{2\pi}
		\frac{\aa k}{2^\aa}
		\frac{\sfB\left (\frac32,\frac{\aa}{2}- \frac12  \right ) }{E^{\frac{\aa}{2} -\frac12}}
		\frac{1}{\sfg_{\aa,k}(E)}
		\zeta_p^{-\frac{\aa}{2}} (1+\xi_p)^{1 -\aa}
		 = \zeta_p^{-\frac{\aa}{2}} (1+\xi_p)^{1 -\aa}.
	\end{split}	
	\]
	Passing to the limit $p \arr \infty$
	in the above
	inequality we arrive at the sought claim.
\end{proof}

\section{Proof of Theorem~\ref{thm:main}}
With all the preparations made above, the proof of the main result turns out to be very short.
By Proposition~\ref{prop:nonrad_fib} we get
\[
\begin{split}
	\liminf_{E\arr 0^+}\frac{\cNE(\Op)}{\sfg_{\aa,k}(E)}
	& =
	\liminf_{E\arr 0^+}\frac{\cNE(\sfF_0)}{\sfg_{\aa,k}(E)},\\
	\limsup_{E\arr 0^+}\frac{\cNE(\Op)}{\sfg_{\aa,k}(E)}
	& =
	\limsup_{E\arr 0^+}\frac{\cNE(\sfF_0)}{\sfg_{\aa,k}(E)}.
\end{split}
\]
and furthermore, applying Propositions~\ref{prop:lower_bnd} and~\ref{prop:upper_bnd} we conclude that
\[
	\lim_{E\arr 0^+}\frac{\cNE(\Op)}{\sfg_{\aa,k}(E)} = 1,
\]
which is our main result. \qed

\section{Discussion}

Let us briefly outline a few possible extensions of the results obtained here.
Recall first that a domain $\Omg_1\subset\dR^3$ is said to be a \emph{local perturbation} of the
domain $\Omg_2\subset\dR^3$, if $\Omg_1\sm K = \Omg_2\sm K$ for a compact set $K\subset\dR^3$.
Without much effort, our main result in Theorem~\ref{thm:main}
allows for an extension to all local perturbations of the generalized, radially symmetric parabolic layers.
It is worth to stress that such perturbed domains need not necessarily be radially symmetric themselves.

Furthermore, one can introduce into the present model an \emph{Aharonov-Bohm-type magnetic field} along the axis of the layer  in the spirit of~\cite{EK18, KLO17}. In our geometric setting, we expect that such a singular magnetic field can neither `switch off' the infiniteness of the discrete spectrum nor will it modify the principal term in the spectral asymptotics. Note that this conjectured behaviour would be in sheer contrast to that of the conical layers~\cite{KLO17}.

Apparently, an analysis similar to the present one can also be performed for the self-adjoint three-dimensional
\emph{Schr\"odinger operator with an attractive $\dl$-interaction} of constant strength supported on the (generalized) paraboloid $\Sg$ in~\eqref{eq:Sigma}. Taking the results of~\cite{BEL14, LO16, OP17} into account,
one may conjecture that the spectral asymptotics will be the same as in Theorem~\ref{thm:main}.
However, proof of such a claim might be technically more involved than for the Dirichlet layers.

Finally, it is worth noting that the analogous spectral problem can be also considered for the self-adjoint
\emph{Robin Laplacian} in the unbounded domain $\{(\xx,x_3)\in\dR^3\colon x_3 > f(|\xx|)\}$, lying above the surface $\Sg$. With a reference to~\cite{BPP17, P16} we expect the infiniteness of the discrete spectrum for the Robin boundary condition with a negative parameter, in other words, for an attractive boundary interaction. In view of~\cite{BPP17}, however, the principal term in the spectral asymptotics might not be the same as in Theorem~\ref{thm:main}.

\begin{appendix}

\section{The signed curvature of $\G$} \label{app:curvature}
%
In this appendix, we analyze properties of the signed curvature $\gg$ of the curve $\G$ defined in~\eqref{eq:signedcurv}.
This material can be seen as an exercise in the differential geometry, however, the claims we need are scattered and not easy to find in textbooks. Recall that the curve $\G$ is parametrized by $\cdR_+\ni s \arr (\phi(s),f(\phi(s)))$,
where the increasing function $\phi\colon\cdR_+\arr\cdR_+$ fulfils $\phi(0) = 0$ and
satisfies the ordinary differential equation~\eqref{eq:ODE_phi}.
In the remaining part of this appendix, all the functions depend on $s$ and their derivatives are taken with respect to that
variable. For the sake of brevity, the indication of the dependence on $s$ is occasionally dropped.

First, we formulate and prove an
auxiliary lemma on the second principal curvature $\kp_2$ of $\Sg$,
explicitly given in~\eqref{eq:curvatures}.
\begin{lem}\label{lem:phi_bnd}
	Let $\phi\colon\cdR_+\arr\cdR_+$ be the solution  of~\eqref{eq:ODE_phi} with $\phi(0) = 0$.
	Then the function
	\[
		\cdR_+ \ni p\mapsto\sup_{s\in [p,\infty)}\frac{\df(\phi(s))\dph(s)}{\phi(s)}
	\]
	is bounded and vanishes as $p\arr \infty$.
\end{lem}	
\begin{proof}
	First, we observe that the function
	$\dR_+\ni s	\mapsto \frac{\df(\phi(s))\dph(s)}{\phi(s)}$
	is $\cC^\infty$-smooth on $\dR_+$.
	Moreover, using that $\phi(0) = 0$, $\dph(0) = 1$
	and the Taylor expansion $\df(x) = \ddf(0)x + o(x)$ as $x\arr 0^+$
	we get
	\[
		\lim_{s\arr 0^+}\frac{\df(\phi)\dph}{\phi} = \ddf(0)
		\and
		\lim_{s\arr \infty}\frac{\df(\phi)\dph}{\phi}
		=
		\lim_{s\arr \infty}
		\frac{1}{\phi}\left(\frac{\df^2(\phi)}{1+\df^2(\phi)}\right)^{\frac12}
		=
		\lim_{s\arr \infty}\frac{1}{\phi} = 0.
	\]	
	The above two limits and smoothness
		of $\frac{f(\phi)\dph}{\phi}$ yield
		the claims.
\end{proof}	
Next, we prove a proposition,
on the asymptotic behaviour
of $\phi$ and its derivatives
up to the third, in the limit $s\arr\infty$.
\begin{prop}\label{prop:phi_bnd}
	The solution $\phi\colon\cdR_+\arr\cdR_+$   of~\eqref{eq:ODE_phi} with $\phi(0) = 0$
	has the following properties.
  	\begin{multicols}{2}
		\begin{itemize}
			\item [{\rm (i)}]
			$\lims
			s^{-\frac{1}{\aa}} \phi =	\limphI $.
			\item [{ \rm (ii)}]	$\lims
			s^{\frac{\aa-1}{\aa}} \dph  =
			\limphII$.
		\end{itemize}
		\begin{itemize}
			\item [{ \rm (iii)}]	
			$\lims
			s^{\frac{2\aa-1}{\aa}} \ddph
			= \limphIII$.
			\item [{\rm (iv)}] $\lims s^{\frac{3\aa-1}{\aa}}\dddph = \limphIV$.
		\end{itemize}	
	\end{multicols}
\end{prop}	
\begin{proof}
	Notice that we obviously have
	$\lims\phi(s) = \infty$,
	because the map $\cdR_+\ni s\mapsto (\phi(s), f(\phi(s))$ is a re-parametrization of the curve $\cdR_+\ni x\mapsto (x,f(x))$.
	The differential equation~\eqref{eq:ODE_phi} implies
	that on the interval $[R,\infty)$
	the following estimates hold:
	\[
		\aa k\phi^{\aa-1}\dph \le 1
		\and
		\big(1 + \aa k\phi^{\aa-1}\big)\dph \ge 1.
	\]
	Integrating the above inequalities on the interval
	$[R,s]$ we get
	\begin{equation}\label{eq:phi_ineqs1}
		k\phi^\aa(s) \le s + \cO(1)\and
		\phi(s) + k\phi^\aa(s) \ge s + \cO(1),
		\qquad \forall\, s \ge R.
	\end{equation}
	Plugging the first inequality
	in~\eqref{eq:phi_ineqs1} into the second, 	
	we obtain
	\begin{equation}\label{eq:phi_ineqs2}
		k\phi^\aa(s) \ge s - \left(\frac{s}{k}\right)^{\frac{1}{\aa}}
		+ \cO(1),
		\qquad
		\forall\, s \ge R.
	\end{equation}
	Combining the first inequality in~\eqref{eq:phi_ineqs1} with~\eqref{eq:phi_ineqs2}
	and using the assumption  $\aa > 1$ we get the limit in~{\rm (i)}.
	Furthermore, the limit in~{\rm (ii)} can be
	shown as follows,
	\[
	\begin{split}
		\lims s^{\frac{\aa-1}{\aa}}\dot\phi
		& =
		\lims s^{\frac{\aa-1}{\aa}}\big(1+\aa^2 k^2 \phi^{2\aa-2}\big)^{-\frac12}
		=
		\lims \left(s^{-\frac{2(\aa-1)}{\aa}} +
				\aa^2 k^2 s^{-\frac{2(\aa-1)}{\aa}}
				\phi^{2\aa-2}\right)^{-\frac12} \\
		& =
		\lims \frac{1}{\aa k s^{-\frac{\aa-1}{\aa}}\phi^{\aa-1}} = \frac{ k^{-\frac{1}{\aa}}}{\aa}.
	\end{split}
	\]
	The  differential equation~\eqref{eq:ODE_phi}
	can be alternatively written as
	\begin{equation}\label{eq:dot_phi}
		\dph(s) = \frac{1}{\big(1+\df^2(\phi(s))\big)^{\frac12}}.
	\end{equation}
	Differentiating the left and right hand sides
	of equation~\eqref{eq:dot_phi}, we express $\ddph$ as follows,
	\begin{equation}\label{eq:ddot_phi}
		\ddph = -\frac{\df(\phi) \ddf(\phi) \dph}{(1+\df^2(\phi))^{\frac32}}
				= - \frac{\df(\phi) \ddf(\phi)}{(1+\df^2(\phi))^2}.	
	\end{equation}
	Hence, on the interval
	$[R,\infty)$, we have
	\begin{equation*}
		\ddph = - \frac{\aa^2k^2(\aa-1)\phi^{2\aa -3}}{ \big(1+\aa^2 k^2 \phi^{2\aa-2}\big)^2 }.
	\end{equation*}
	Eventually, using {\rm (i)} we get
	\[
		\lims s^{\frac{2\aa-1}{\aa}}\ddph
			= 	- \lims\frac{\aa-1}{\aa^2 k^2 s^{-\frac{2\aa-1}{\aa}}  \phi^{2\aa-1}}
			=  \frac{\aa-1}{\aa^2 k^2 k^{-\frac{2\aa-1}{\aa}}}	
			=  -\frac{\aa-1}{\aa^2}k^{-\frac{1}{\aa}},	
	\]
	and in this way the limit in~{\rm (iii)} is also obtained.
	
	Differentiating the left and the right hand sides
	of~\eqref{eq:ddot_phi}, we express $\dddot\phi$ as follows,
	\[
	\begin{split}
		\dddph& =
			- \frac{(\ddf^2(\phi) + \df(\phi)\dddf(\phi))
			(1+\df^2(\phi))\dph
			- 4\df^2(\phi)\ddf^2(\phi)\dph}{(1+\df^2(\phi))^3}	\\
		& =
			\frac{3\df^2(\phi)\ddf^2(\phi)
				-\ddf^2(\phi) - \df(\phi)\dddf(\phi)
			- \df^3(\phi)\dddf(\phi)}
						{(1+\df^2(\phi))^{\frac72}}.
	\end{split}
	\]
	The latter yields that on the interval
	$[R,\infty)$
	\[
	\begin{split}
		\dddph
		& =
		\frac{
		3\aa^4(\aa-1)^2 k^4\phi^{4\aa-6}
		-
		\aa^2(\aa-1)(2\aa-3)k^2\phi^{2\aa-4}
		- \aa^4(\aa-1)(\aa-2)k^4\phi^{4\aa-6}
		}{\big(1+ \aa^2 k^2\phi^{2\aa-2}\big)^{\frac72}}\\
		& =
		\frac{
		\aa^4(\aa-1)	(2\aa-1)k^4\phi^{4\aa-6}
		-
		\aa^2(\aa-1)(2\aa-3)k^2\phi^{2\aa-4}
			}{\big(1+ \aa^2 k^2\phi^{2\aa-2}\big)^{\frac72}}.
	\end{split}
	\]
	Again using~{\rm (i)} we obtain
	\[
	\begin{split}
		\lims
		s^{\frac{3\aa-1}{\aa}}\dddot\phi
		& =
		\lims
		s^{\frac{3\aa-1}{\aa}}
			\frac{
				\aa^4(\aa-1)(2\aa-1)k^4\phi^{4\aa-6}
					-
				\aa^2(\aa-1)(2\aa-3)k^2\phi^{2\aa-4}
			}{\big(1+ \aa^2 k^2\phi^{2\aa-2}\big)^{\frac72}}\\
		& =
		\lims
		s^{\frac{3\aa-1}{\aa}}
		\frac{\aa^4(\aa-1)(2\aa-1)k^4
			\phi^{4\aa -6}}{\aa^7k^7\phi^{7\aa-7}}
		=
		\lims
		\frac{(\aa-1)(2\aa-1)}{\aa^3 k^3	s^{-\frac{3\aa-1}{\aa}}\phi^{3\aa-1}}\\
		& =
		\frac{(\aa-1)(2\aa-1)}{\aa^3 k^3 k^{-\frac{3\aa-1}{\aa}}}
		=
		\frac{(\aa-1)(2\aa-1)}{\aa^3} k^{-\frac{1}{\aa}},
	\end{split}
	\]
	by which the limit in {\rm (iv)} is also shown.
\end{proof}
Recall that the signed curvature of the curve $\G$ is given by the formula
\begin{equation}\label{eq:signed_curv_app}
	\gg = \ddf(\phi)\dph^3.
\end{equation}
Finally, we prove a claim about the asymptotic behaviour
of $\gg$ and its derivatives
up to the second order, in the limit $s\arr\infty$.
\begin{prop}\label{prop:gamma}
	Let the signed curvature $\gg\colon\cdR_+\arr\dR$
	be as
	in~\eqref{eq:signed_curv_app}.
	Then there exist $g_j\in\dR$, $j=0,1,2$,
	such that:
	\begin{myenum}
		\item	$\lims s^{\frac{2\aa-1}{\aa}} \gg
		 = g_0$,
		\item $\lims s^{\frac{3\aa-1}{\aa}} \dot\gg = 	g_1$,
		\item $\lims s^{\frac{4\aa-1}{\aa}} \ddot\gg
		=	g_2$.
	\end{myenum}	
\end{prop}	
\begin{proof}
	The first and the second derivatives
	of the signed curvature $\gg$ are given by
	\[
	\begin{split}
		\dot\gg & = 3\dot\phi^2\ddot\phi\ddf(\phi)
		+\dot\phi^4 \dddf(\phi),\\
		\ddot\gg & =
		6\dph\ddph^2\ddf(\phi)
		+3\dph^2\dddph\ddf(\phi)
		+7\dph^3\ddph\dddf(\phi)
		+\dph^5 \ddddf(\phi).\\
	\end{split}	
	\]
	Hence, using the notation
	$\kp := \aa(\aa-1)k$
	we infer that on the interval
	$[R,\infty)$ the following relations hold:
	\begin{equation*}\label{key}
	\begin{split}
		\gg
		& = \kp\phi^{\aa-2}\dph^3,\\
		\dot\gg & =
		\kp\left [
		3\phi^{\aa-2}\dph^2\ddph
		+(\aa-2)\phi^{\aa-3}\dph^4\right],\\
		\ddot\gg& =
		\kp\Big[
		\phi^{\aa-2}
		\big(6\dph\ddph^2	+3\dph^2\dddph\big)
		 +
		7(\aa-2)\phi^{\aa-3}
		\dph^3\ddph
		 +(\aa-2)(\aa-3)\phi^{\aa-4}
		\dph^5\Big].	
	\end{split}
	\end{equation*}
	Eventually, existence of finite limits
	in (i)-(iii) directly follows
	from Proposition~\ref{prop:phi_bnd}\,(i)-(iv).
\end{proof}
\end{appendix}

\section*{Acknowledgment}
The authors acknowledge the support by the grant
No.~17-01706S of the Czech Science Foundation (GA\v{C}R).

\newcommand{\etalchar}[1]{$^{#1}$}

\end{document}

%% file: commands.tex
\newcommand\distance{$\dist(\xx, E) := \inf_{\yy\in E}|\xx-\yy|_{\dR^d}$ 
	is the distance between a point $\xx\in\dR^d$ and 
	a set $E\subset\dR^d$.}

\newcommand\pot{\left(\frac{k}{s} \right )^{\frac{2}{\aa}}}

\newcommand\df{\dot{f}}
\newcommand\ddf{\ddot{f}}
\newcommand\dddf{\dddot{f}}
\newcommand\ddddf{f^{(4)}}

\newcommand\limphI{ k^{-\frac{1}{\aa}}}
\newcommand\limphII{\frac{ k^{-\frac{1}{\aa}}}{\aa}}
\newcommand\limphIII{-\frac{\aa-1}{\aa^2}k^{-\frac{1}{\aa}}}
\newcommand\limphIV{\frac{(\aa-1)(2\aa-1)}{\aa^3} k^{-\frac{1}{\aa}}}

\newcommand\cNE{\cN_{\lme - E}}

\newcommand\dph{\dot{\phi}}
\newcommand\ddph{\ddot{\phi}}
\newcommand\dddph{\dddot{\phi}}

\newcommand\lims{\lim_{s\arr \infty}}

\renewcommand\gg{\gamma}

\newcommand\frh{\mathfrak{h}}
\newcommand\fri{\mathfrak{i}}
\newcommand\tm{\times}
\newcommand\nb{\nabla}

\newcommand{\xx}{{\boldsymbol{x}}}
\newcommand{\yy}{{\boldsymbol{y}}}

\newcommand{\dist}{\mathop{\mathrm{dist}}\nolimits}

\renewcommand{\aa}{\alpha}

\theoremstyle{definition}

\usepackage[normalem]{ulem}

\makeatletter
\newcommand{\vast}{\bBigg@{3}}
\newcommand{\Vast}{\bBigg@{5}}

\newcommand\cdR{\ov{\dR}}

\usepackage{scrextend}
\changefontsizes{10.34pt}

\newcommand\cyl{\mathrm{cyl}}
\newcommand\lme{{E_\infty}}

\newcommand{\eg}{{\it e.g.}\,}
\newcommand{\ie}{{\it i.e.}\,}
\newcommand{\cf}{{\it cf.}\,}

\renewcommand\and{\qquad\text{and}\qquad}

\newcommand\sm{\setminus}
\newcommand\Op{\sfH}

\newcommand\dl{\delta}
\newcommand\Dl{\Delta}
\newcommand\frf{\mathfrak{f}}

\newcommand{\comm}[1]{}

\newcommand\G{\Gamma}

\renewcommand\aa{\alpha}
\newcommand\lm{\lambda}
\newcommand\s{\sigma}
\newcommand\p{\partial}

\newcommand\Omg{\Omega}
\renewcommand\tt{\theta}

\newcommand\kp{\kappa}

\newcommand\ii{{\mathsf{i}}}

\newcommand\one{\mathbbm{1}}

\renewcommand\Re{{\rm Re}\,}

\newcommand\arr{\rightarrow}

\newcommand\Sg{\Sigma}

\newcommand\sd{\sigma_{\rm d}}
\newcommand\sess{\sigma_{\rm ess}}

\newcommand\dd{{\mathsf{d}}}

\newcounter{counter_a}
\newenvironment{myenum}{\begin{list}{{\rm(\roman{counter_a})}}%
{\usecounter{counter_a}
\setlength{\itemsep}{1.ex}\setlength{\topsep}{0.8ex}
\setlength{\leftmargin}{5ex}\setlength{\labelwidth}{5ex}}}{\end{list}}

\usepackage[latin1]{inputenc}
\usepackage[T1]{fontenc}

\numberwithin{figure}{section}
\numberwithin{equation}{section}
\theoremstyle{plain}
\newtheorem*{thm*}{Theorem}
\newtheorem{thm}{Theorem}[section]

\newtheorem{lem}[thm]{Lemma}
\newtheorem{prop}[thm]{Proposition}

\theoremstyle{remark}
\newtheorem{remark}[thm]{Remark}
\theoremstyle{plain}

\newcommand\ov{\overline}

\def\ov{\overline}

   \def\sE{{\mathfrak E}}   \def\sF{{\mathfrak F}}

   \def\sQ{{\mathfrak Q}}   
   \def\sT{{\mathfrak T}}

\def\frs{{\mathfrak s}}
\def\frt{{\mathfrak t}}
\def\fru{{\mathfrak u}}

      \def\dC{{\mathbb C}}
      
      \def\dI{{\mathbb I}}
      
   \def\dN{{\mathbb N}}   
      \def\dR{{\mathbb R}}
\def\dS{{\mathbb S}}      
      
   \def\dZ{{\mathbb Z}}

   \def\sfB{{\mathsf B}}   
      \def\sfF{{\mathsf F}}
   \def\sfH{{\mathsf H}}

   \def\sfT{{\mathsf T}}   \def\sfU{{\mathsf U}}
\def\sfV{{\mathsf V}}

   \def\cB{{\mathcal B}}   \def\cC{{\mathcal C}}
   \def\cE{{\mathcal E}}   \def\cF{{\mathcal F}}
\def\cG{{\mathcal G}}      
\def\cJ{{\mathcal J}}   \def\cK{{\mathcal K}}   \def\cL{{\mathcal L}}
   \def\cN{{\mathcal N}}   \def\cO{{\mathcal O}}
\def\cP{{\mathcal P}}      
\def\cS{{\mathcal S}}   \def\cT{{\mathcal T}}   
      \def\cX{{\mathcal X}}

      \def\sff{{\mathsf f}}
\def\sfg{{\mathsf g}}   \def\sfh{{\mathsf h}}

\newcommand{\dom}{\mathrm{dom}\,}

\makeatletter
\def\section{\@startsection{section}{1}\z@{.9\linespacing\@plus\linespacing}%
	{.7\linespacing} {\fontsize{13}{14}\selectfont\bfseries\centering}}

\def\subsection{\@startsection{subsection}{1}\z@{.9\linespacing\@plus\linespacing}%
	{.3\linespacing} {\fontsize{10}{12}\selectfont\bfseries}}	
	
\def\paragraph{\@startsection{paragraph}{4}%
	\z@{0.3em}{-.5em}%
	{$\bullet$ \ \normalfont\itshape}}

\@addtoreset{equation}{section}
\makeatother